\documentclass[11pt,reqno]{amsart}

\usepackage[centering,a4 paper]{geometry}
\usepackage[utf8]{inputenc}
\usepackage{amsmath,amsaddr,amssymb,amsthm}
\usepackage[inline]{enumitem}
\usepackage{thmtools}
\usepackage{thm-restate}
\usepackage{hyperref,xcolor,graphicx,float}
\usepackage[capitalize,nameinlink,sort]{cleveref}

\allowdisplaybreaks
\crefname{enumi}{item}{items}
\crefname{enumi}{Item}{Items}

\makeatletter
\newcommand{\symitem}[2]{%
  \item[#1]%
  \phantomsection                    
  \protected@edef\@currentlabel{#1}  
  \label[enumi]{#2}                  
}
\makeatother

\usepackage{pgf,tikz,tikz-cd}
\usetikzlibrary{arrows,arrows.meta,automata, calc, chains, decorations, shapes, positioning,decorations.pathmorphing}

\hypersetup{hidelinks}

\declaretheorem[numberwithin=section]{theorem}
\declaretheorem[sibling=theorem]{lemma,corollary,proposition}
\declaretheorem[sibling=theorem,style=definition]{example,definition,remark}

\declaretheorem[sibling=theorem,style=definition,refname={question},Refname={Question}]{question}

\usepackage{mathtools}
\usepackage[normalem]{ulem}
\newcommand{\seqnum}[1]{\href{https://oeis.org/#1}{\rm \underline{#1}}}
\newcommand{\btau}{\boldsymbol{\tau}}
\newcommand{\cA}{\mathcal{A}}
\newcommand{\cB}{\mathcal{B}}
\newcommand{\cC}{\mathcal{C}}
\newcommand{\cL}{\mathcal{L}}
\newcommand{\cR}{\mathcal{R}}
\newcommand{\cS}{\mathcal{S}}
\newcommand\NN{\mathbb{N}}
\newcommand\RR{\mathbb{R}}

\newcommand\faset{\mathcal{L}}
\newcommand{\infw}[1]{{\bf{#1}}}

\usepackage{todonotes}
\title{Factor-balancedness, linear recurrence, and factor complexity}
\date{\today}

\author{Bastián Espinoza and Manon Stipulanti}
\address{Department of Mathematics, University of Liège, Allée de la Découverte 12, 4000 Liège, Belgium}
\email{baespinoza@uliege.be, m.stipulanti@uliege.be}

\author{Pierre Popoli}
\address{School of Computer Science, University of Waterloo, Waterloo, ON N2L 3G1, Canada}
\email{pierre.popoli@uwaterloo.ca}

\begin{document}

\begin{abstract}
In the study of infinite words, various notions of balancedness provide quantitative measures for how regularly letters or factors occur, and they find applications in several areas of mathematics and theoretical computer science.
In this paper, we study factor-balancedness and uniform factor-balancedness, making two main contributions.
First, we establish general sufficient conditions for an infinite word to be (uniformly) factor-balanced, applicable in particular to any given linearly recurrent word.
These conditions are formulated in terms of $\mathcal{S}$-adic representations and generalize results of Adamczewski on primitive substitutive words, which show that balancedness of length-2 factors already implies uniform factor-balancedness. 
As an application of our criteria, we characterize the Sturmian words and ternary Arnoux--Rauzy words that are uniformly factor-balanced as precisely those with bounded weak partial quotients.
Our second main contribution is a study of the relationship between factor-balancedness and factor complexity.
In particular, we analyze the non-primitive substitutive case and construct an example of a factor-balanced word with exponential factor complexity, thereby making progress on a question raised in 2025 by Arnoux, Berthé, Minervino, Steiner, and Thuswaldner on the relation between balancedness and discrete spectrum.
\end{abstract}

\maketitle

\noindent\textbf{Keywords}: Combinatorics on words, balancedness, factor-balancedness, uniform factor-balancedness, discrepancy, Sturmian words, Arnoux--Rauzy words, linearly recurrent words, substitutions, decisiveness, factor complexity, Toeplitz words. 

\medskip 

\noindent\textbf{2020 Mathematics Subject Classification}: 68R15. 

\medskip 

\noindent\textbf{Acknowledgments}. We thank Boris Adamczewski, Shigeki Akiyama, Pierre Arnoux, Valérie Berthé, and Jeffrey Shallit for insightful discussions all along this project. 

Bastián Espinoza is supported by ULiège's Special Funds for Research, IPD-STEMA Program. 
Pierre Popoli is supported by NSERC grant 2024-03725.
Manon Stipulanti is an FNRS Research Associate supported by the Research grant 1.C.104.24F. 

\setcounter{tocdepth}{1} 

\tableofcontents


\section{Introduction}

The study of \emph{balanced} infinite \emph{words}---also called \emph{sequences}, \emph{strings}, or \emph{streams}---originates in the pioneering work of Morse and Hedlund~\cite{Morse-Hedlund-1938,mh40} in the late 1930s.
Their research introduces balancedness as a fundamental property linking the combinatorics of the so-called \emph{Sturmian} words with seemingly distant areas such as billiard dynamics in the square and the arithmetic of continued fractions. 
This perspective laid the foundations for what has since become a recurrent theme in combinatorics on words (for instance, see~\cite{DurandPerrinBook, PytheasFogg2002, Lothaire2002, Thuswaldner2020}) and also across several other fields. 
In the theory of aperiodic order, balancedness is an important regularity property of Delone sets, tilings, and quasicrystals, providing a quantitative measure of how uniformly points or tiles are spread; for instance, see~\cite{BaakeGrimm2013}.
It is also a recurring notion in uniform distribution theory~\cite{KN1974}.
In topological dynamics on the Cantor set and ergodic theory, the Gottschalk--Hedlund theorem relates the balancedness of clopen sets to cohomological properties and to the existence of topological eigenvalues~\cite{DurandPerrinBook,Gottschalk-Hedlund}. 
From a computational viewpoint, balancedness is connected to certain optimization and operations research problems; for instance, see~\cite{AGH2000,Tijdeman1980}. 
In arithmetic and number theory, it appears under the name of bounded remainder sets and is closely related to continued fractions (for instance see~\cite{PytheasFogg2002}) as well as to open problems such as Fraenkel's conjecture~\cite{Fraenkel1973}. 
Originally stated in the field of number theory~\cite{Fraenkel1973}, this conjecture can be reformulated combinatorially as follows: for each $k \ge 3$, there is only one $1$-balanced infinite word over a $k$-letter alphabet with different letter frequencies, up to a permutation of the letters. 
This conjecture is verified for small values of $k$, see~\cite{AGH2000,Tijdeman2000bis,Tijdeman2000} for $k\in\{3,4,5,6\}$, and some particular families of words such as episturmian words, see~\cite{PV2007}.
Finally, in the context of billiard dynamics, the complete unbalancedness is linked to the notion of topological weak mixing.

In combinatorics on words and symbolic dynamics, balancedness is a quantitative measure of how uniformly factors occur in a given infinite word.
Originally, Morse and Hedlund~\cite{Morse-Hedlund-1938,mh40} define an infinite word $\infw{x}$ to be what is widely called now\footnote{We warn the reader that this property is sometimes plainly called \emph{balancedness}. Here, we use the terms \emph{letter-balanced} and \emph{factor-balanced} to mean $C$-letter-balanced and $C$-factor-balanced, respectively, for some constant $C$, which are introduced later on.} \emph{$1$-balanced} if, for any letter $a$ and all factors $u,v$ of $\infw{x}$ with the same length, we have $\bigl||u|_a-|v|_a\bigr|\le 1$ (where the notation $|w|_a$ stands for the number of occurrences of the letter $a$ in the word $w$). 
The notion of $1$-balancedness is then extended by replacing the constant $1$ with an arbitrary integer $C \ge 1$, leading to the notion of \emph{$C$-letter-balancedness}.
We note that this is equivalent to a bounded discrepancy condition (with a possibly different constant), see~\cite{Ada03} or \cref{lem:balance=>discrepancy} below; for a more general reference on the discrepancy theory, we point the reader to~\cite{KN1974}.
It is also equivalent to having a bounded abelian factor complexity as pointed out in~\cite{RichommeSaariZamboni-2011}. 
The seminal paper of Rauzy~\cite{Rauzy82} uses letter-balancedness as a key property to construct a fractal set associated with the \emph{Tribonacci} substitution. 
This work extends, to a two-dimensional setting, the connection between arithmetic geometry and combinatorics on words, which was already present in the Sturmian case. Rauzy's paper initiated a large body of subsequent research.
In particular, a long series of works aim at generalizing this interplay to broader families of infinite words. 
This line of research includes the introduction of Arnoux–Rauzy words~\cite{ArnouxRauzy1991}, the refutation of the conjecture asserting their letter-balancedness in all cases~\cite{CFZ2000}, the study of Lyapunov exponents associated with these systems~\cite{AD2019}, and results showing that they are almost surely balanced and admit geometric representations~\cite{bertheJEMS}. 
Further extensions were developed for other families, such as Brun words~\cite{bertheJEMS}.
Recently, Berthé, Cecchi--Bernales, and Espinoza establish necessary and sufficient conditions for letter-balancedness of primitive substitutive words in terms of incidence matrices and so-called extension graphs, see~\cite[Proposition 6.1]{BertheCecchiEspinoza}.

Instead of considering letters, one may turn to factors.
Given an integer $C\ge 1$ and an infinite word $\infw{x}$, a finite word $w$ is \emph{$C$-balanced in $\infw{x}$} if, for all factors $u,v$ of $\infw{x}$ of the same length, the difference between the numbers of occurrences of $w$ in $u,v$ is bounded by $C$.
If each factor of $\infw{x}$ is $C$-balanced in $\infw{x}$ for some $C$, then $\infw{x}$ is said to be \emph{factor-balanced}.
When the bound $C$ is uniform, that is, independent of the chosen factors of $\infw{x}$, we speak of \emph{uniform factor-balancedness}. 
The study of factor-balancedness and its uniform version is the main focus of this paper. 

Contrary to letter-balancedness, which depends only on the abelianization of words, factor-balancedness depends on finer combinatorics and therefore cannot be reduced to the study of letters in general.
However, Adamcwzeski~\cite{Ada03} proves that the factor-balancedness of primitive substitutive words is equivalent to balancedness for length-$2$ factors.
Related to famous families of words, Fagnot and Vuillon~\cite{FV:2002} show that Sturmian words are factor-balanced and give conditions under which they are even uniformly factor-balanced.
In addition, many works are dedicated to investigating the links between letter-balancedness and factor-balancedness of different lengths.  
For instance, Berthé and Cecchi--Bernales~\cite{BertheCecchi} prove that letter-balancedness implies factor-balancedness in dendric words.
This family of words has gained interest in recent years and generalizes Arnoux--Rauzy words and codings of interval exchange transformations; for instance see~\cite{Bertheetal2015a,Bertheal2015b,Bertheal2015c,Gheeraert1,Gheeraert2}. 
Later on, the result of Berthé and Cecchi--Bernales was generalized by Berthé et al.~\cite[Corollary 5.5]{BCBDLPP2021} to primitive unimodular proper $\cS$-adic subshifts.
More recently, Arnoux et al.~\cite{NonStationaryMarkov} connect factor-balancedness with the Pisot conjecture, which asserts that any irreducible Pisot substitution generates a symbolic dynamical system with pure discrete spectrum; see~\cite{ABBLS2015} for a general overview of this conjecture. 
In this context, Arnoux et al.~\cite[Question 18.2]{NonStationaryMarkov} ask which are the factor-balanced infinite words generating a symbolic system with a discrete spectrum.
Regarding the relationship between the original notion of letter-balancedness and its extension to factor-balancedness, we also mention the works of Poirier and Steiner~\cite{PoirierSteiner} and that of Quéffelec~\cite{Queffelec-2010}.

Although uniform factor-balancedness is stronger than factor-balancedness, Adamczewski's results~\cite{Ada03,Ada04} imply that any fixed point of a substitution whose length-$2$ sliding block code is a primitive \emph{Pisot} substitution (i.e., with second-largest eigenvalue in modulus (strictly) less than $1$) satisfies this property. 
As a consequence, many classical infinite words belong to this class, e.g., the Tribonacci word, and more generally, any substitutive Arnoux-Rauzy word~\cite {ArnouxIto2001}.
To the best of our knowledge, this property explicitly appeared only two decades later as a computational hypothesis in~\cite{CDOPSS2025}, motivated by the use of the theorem-prover \texttt{Walnut}~\cite{Walnut-2023}, which is built on the connection between first-order logic and automata theory.
From this hypothesis, Couvreur et. al.~\cite{CDOPSS2025} prove that the two-dimensional generalized abelian complexity of an automatic uniform factor-balanced word is regular, in the sense of Allouche and Shallit~\cite{AS:1992}. 
Furthermore, authors of~\cite{CDOPSS2025} also prove the following explicit result about the Tribonacci word, generalizing that of~\cite{RSZ:2010} concerning letter-balancedness.

\begin{theorem}[\cite{CDOPSS2025}]
\label{thm: Tribo is 2 UFB}
The Tribonacci word is uniformly factor-balanced with constant $2$. 
\end{theorem}

However, Adamczewski's general theorem appears to be largely underutilized for factor-balancedness in the literature. 
This became apparent when the authors of~\cite{CDOPSS2025} presented their work to audiences in combinatorics on words, dynamical systems, and numeration systems, where uniform factor-balancedness (and in particular Theorem \ref{thm: Tribo is 2 UFB}) often produced some surprise.
This observation constitutes the main motivation for the present work. 
Our goal is to bring uniform factor-balancedness back into focus by revisiting known results, clarifying its scope and significance, and extending the current understanding of this property, as well as providing some new general results on factor-balancedness. 

\subsection{Our contributions}
\label{sec: our contributions}
In this paper, we contribute to the theory of factor-balancedness and uniformly factor-balancedness in two different ways.






\textbf{First}, in~\cref{sec:UFB&LR}, we provide conditions under which linearly recurrent words are (uniformly) factor-balanced.
We give our result here in~\cref{thm:char_balance_LR} but we refer the reader to~\cref{sec:UFB&LR} for the necessary definitions and notation that appear in the statement. Note that a congenial sequence as stated below always exists for a given linearly recurrent word.

\begin{restatable}{theorem}{maintwo}
\label{thm:char_balance_LR}
    Let $\infw{x}$ be a linearly recurrent word and $(\tau_n,\, a_n)_{n \ge 0}$ be a congenial sequence generating $\infw{x}$ satisfying Conditions \ref{defi:L1}, \ref{defi:L2}, and \ref{defi:L3bis}.
    Let $\infw{x^{(n)}}$ denote the limit of $(\tau_{n,m}(a_m))_{m > n}$.
    Suppose that there exists a sequence $(B_n)_{n\geq 0}$ such that the word $\infw{x^{(n)}}$ is $B_n$-letter-balanced. Then the word $\infw{x}$ is factor-balanced. 
    Moreover, if the sequence $(B_n)_{n\geq 0}$ is bounded, then the word $\infw{x^{(n)}}$ is uniformly factor-balanced. 
\end{restatable}

Our approach to~\cref{thm:char_balance_LR} is based on the rich theory of $\cS$-adic representations of linearly recurrent words, and is explained in~\cref{subsec:LR_as_Sadic}.
Another crucial ingredient of our approach is the notion of decisiveness in~\cref{sec:decisiveness}. This allows us to reduce in~\cref{subsec:charact_LR_balance} the problem of studying factor-balancedness to that of letter-balancedness in an adequate desubstitution of the original word.

\cref{thm:char_balance_LR} can be applied to any primitive substitutive word, since such a word is linearly recurrent \cite{Durand-1999}. In this case, we obtain in~\cref{subsec:charact_LR_balance} an equivalence between uniform factor-balancedness and balancedness of length-$2$ factors. 

\begin{proposition}
\label{pro: Adamczewski for primitive substitutive words} 
Let $\infw{x}$ be the fixed point of a primitive substitution. Then each length-2 factor of $\infw{x}$ is balanced in $\infw{x}$ if and only if the word $\infw{x}$ is uniformly factor-balanced.
\end{proposition}

This phenomenon already appears in the work of Adamczewski, following directly from~\cite[Theorem 22]{Ada03}. 

More importantly, another application of~\cref{thm:char_balance_LR} provides a full characterization of binary and ternary Arnoux--Rauzy words that are uniformly factor-balanced, which we discuss in~\cref{sec:UFB_AR}. All necessary definitions are introduced in~\cref{subsec:ARwords}.
Also note that the binary case, i.e., about Sturmian words, was already considered in~\cite{FV:2002}.

\begin{restatable}{theorem}{mainAR} \label{thm:mainAR} 
   An Arnoux--Rauzy word over an alphabet of at most $3$ letters is uniformly factor-balanced if and only if it has bounded weak partial coefficients. 
\end{restatable}


\textbf{Second}, in~\cref{sec:complexity}, we study links between factor-balancedness and factor complexity. 
We first prove in \cref{UFB&substitutive=>uniformly_recurrent} that substitutive factor-balanced words have linear factor complexity, and in the more specialized \cref{UFB&substitutive=>uniformly_recurrent} we describe a special substitutive structure for these words.
Furthermore, with~\cref{thm: FB word with exp complexity}, we then provide an example of a factor-balanced word with exponential factor complexity. 
This reasoning is based on Toeplitz words, whose construction is recalled in~\cref{sec:Toeplitz}.

\begin{theorem}
\label{thm: FB word with exp complexity}
    There exists an infinite word that is factor-balanced and has exponential-growth factor complexity.
    In particular, the symbolic dynamical system generated by $\infw{x}$ has positive topological entropy and thus cannot have pure discrete spectrum.
\end{theorem}
This makes progress towards a question raised by Arnoux et al. in~\cite[Question 18.2]{NonStationaryMarkov} by showing that not all factor-balanced symbolic system have discrete spectrum.

Finally, the reader will find in~\cref{sec:preliminaries} the necessary background for the paper; in particular, subsections dedicated to definitions about words, substitutions, the close  connection between discrepancy and balancedness, $\cS$-adic representations, Sturmian words, and Arnoux--Rauzy words.
\cref{sec: generalities on UFB} contains generalities on (uniform) factor-balancedness, especially its stability by invertible higher block codes and its relationship to uniform recurrence in infinite words.
Instead of finishing the paper with some open problems, we propose some directly below, in~\cref{sec: open problems}.


\section{Open problems}
\label{sec: open problems}

In this section, we gather some open questions we came up with during the writing of the paper, which motivates further research on the topic of (uniform) factor-balancedness.

In~\cref{sec:UFB&LR}, \cref{thm:char_balance_LR} provides sufficient conditions for a linearly recurrent word to be (uniformly) factor-balanced.
As a natural question, we ask whether there are necessary conditions for this to hold.

\begin{question}
What are the necessary conditions on linearly recurrent words to be (uniformly) factor-balanced?
\end{question}

To answer this question, it appears that some recognizability hypothesis is required. In~\cref{sec:UFB_AR}, we fully characterize uniformly factor-balanced Sturmian and ternary Arnoux--Rauzy words.
One can naturally wonder whether the property extends to any size alphabet.

\begin{question}
Is it possible to characterize which Arnoux--Rauzy words are uniformly factor-balanced over any $d$-ary alphabet, with $d\ge 2$?
\end{question}

Note that three letters were already a limitation in~\cite{BCS2013}, so the general case might be more intricate and new methods might be necessary to tackle the question.






In~\cref{sec:exponential sub complexity}, we explicitly provide an example of a factor-balanced word with an exponential-growth factor complexity. One may wonder the following finer question. 

\begin{question}
Can every exponential-growth factor complexity be attained with factor-balanced infinite words?
\end{question}



Our Toeplitz construction does not allow us to obtain uniformly factor-balanced examples with exponential-growth factor complexity without substantial modifications. 
Moreover, all known examples of uniformly factor-balanced words have linear factor complexity. 
It is therefore natural to ask the following.

\begin{question}
Is it possible to construct a uniformly factor-balanced infinite word with a superlinear-growth factor complexity?
\end{question}

By \cref{UFB&substitutive=>uniformly_recurrent}, no (uniformly) factor-balanced word with super-linear factor complexity can be substitutive. 
Then, one has to use different words, as we did to prove \cref{thm: FB word with exp complexity} with Toeplitz words.


\section{Preliminaries}
\label{sec:preliminaries}

\subsection{Generalities on words}

An \emph{alphabet} is any finite set of elements, themselves called \emph{letters}. A \emph{finite word} $w$ of length $n\in\mathbb{N}$ over the alphabet $\cA$ is an element $w=w_1w_2\cdots w_n$ where each $w_i$ belongs to $\cA$.
We let $\cA^n$ denote the set of length-$n$ words over $\cA$.
We let $\cA^*$ denote the set of finite words over $\cA$, i.e., $\cA^* = \cup_{n\ge 0} \cA^n$, and we let $\varepsilon$ denote the \emph{empty word}.
An \emph{infinite word} $\infw{x}$ is an element of $\cA^{\NN}$ (unless otherwise specified, we start indexing them at $0$).
To distinguish finite from infinite words, we write the latter in bold.

A finite word $v=v_1\cdots v_k$ is a \emph{factor} of another finite word $u=u_1\cdots u_\ell$ if there exists $i\geq 1$ such that $v_1\cdots v_k=u_i\cdots x_{i+k-1}$. If $i=1$, we say that $v$ is a \emph{prefix} of $u$, and if $i=\ell-k+1$, we say that $v$ is a \emph{suffix} of $u$. If $v$ is a factor of $u$, an \emph{occurrence} of $v$ in $u$ is any integer $i$ such that $v_1\cdots v_k=u_i\cdots x_{i+k-1}$. We let $|u|_v$ denote the number of occurrences of $v$ in $u$ (possibly with overlaps). When $u$ is an infinite word, we use the same terminology.
Given an infinite word $\infw{x}$, we also let $\cL_n(\infw{x})$ denote \emph{its set of factors of length $n$} and $\cL(\infw{x})=\bigcup_{n\geq 0}\cL_n(\infw{x})$ its \emph{set of factors}. We let $p_\infw{x}$ denote the \emph{factor complexity} of $\infw{x}$, i.e., the map $p_\infw{x} \colon \NN \to \NN, n \mapsto \# \faset_n(\infw{x})$ that counts the number of factors of $\infw{x}$ of each length.
If $\infw{x}$ can be written as $uvvv\cdots$ for finite words $u,v$, we say that $\infw{x}$ is \emph{ultimately} or \emph{eventually periodic with period $|v|$}; if $u=\varepsilon$, we say that $\infw{x}$ is \emph{(fully) periodic}. If it is not periodic,  $\infw{x}$ is \emph{aperiodic}.
A word $u\in\cA^*$ is a \emph{power (of another word)} if there exist $v \in \cA^*$ and $k \ge 2$ with $u = v^k$; otherwise we say that $u$ is {\em primitive}.
For any integer $P$, a word is \emph{$P$-power-free} if none of its factors is of the form $v^P$ for some non-empty word $v$. 

Consider an infinite word $\infw{x}=(x_n)_n\in \cA^\NN$ and a finite word $u\in \cA^*$.
The \emph{frequency $\mu_\infw{x}(u)$ of $u$ in $\infw{x}$} is defined, if it exists, as the following limit \[\lim_{n\rightarrow + \infty}\frac{|x_0\cdots x_{n-1}|_u}{n}.\] When the limit exists, we say that $\infw{x}$ admits a \emph{frequency for the word $u$}. 
We say that $\infw{x}$ admits a \emph{uniform frequency for the word $u$} if the number of occurrences of $u$ in $x_k\cdots x_{k+n-1}$ divided by $n$ admits a uniform limit in $k$ when $n$ goes to infinity, i.e., there exists a real number $\mu_u$ such that \[\forall \varepsilon>0, \exists n_0\geq 0, \forall n\geq n_0, \forall k, \left| \frac{1}{n}|x_k \cdots x_{n+k-1}|_u-\mu_u \right| \leq \varepsilon.\] Finally, we say that $\infw{x}$ has \emph{uniform frequencies} if all its factors have uniform frequencies.

An infinite word $\infw{x} \in \cA^{\NN}$ is \emph{recurrent} if every factor of $\infw{x}$ appears infinitely often in $\infw{x}$. 
The word $\infw{x}$ is \emph{uniformly recurrent}, or \emph{minimal}, if there exists an integer $k$ such that every factor $w\in \cL(\infw{x})$ occurs in every length-$k$ factor of $\infw{x}$. 
For any infinite word $\infw{x}$, there exists a uniformly recurrent word $\infw{y}$ such that $\cL(\infw{y}) \subseteq \cL(\infw{x})$
(this result follows from the well-known fact that every topological dynamical system contains a minimal sub-system; for instance, see~\cite[Theorem 5.2]{Walters}).

Let $\infw{x} \in \cA^{\NN}$ be an infinite word, and $u \in \cL(\infw{x})\setminus \{\varepsilon\}$. 
A \emph{return word to $u$ in $\infw{x}$} is a word $w$ such that $wu\in \cL(\infw{x})$ and $u$ occur exactly twice in $wu$ (once as a prefix and once as a suffix). 
We let $\mathcal{R}_{\infw{x}}(u)$ denote the set of returns to $u$ in $\infw{x}$. 
Then, $\infw{x}$ is uniformly recurrent if, for every $u \in \cL(\infw{x})\setminus \{\varepsilon\}$, the set $\mathcal{R}_{\infw{x}}(u)$ is nonempty and finite. Moreover, we say that $\infw{x}$ is \emph{linearly recurrent (with constant $L\ge 0$)} if, for all $u \in \cL(\infw{x})\setminus \{\varepsilon\}$ and all $w \in \mathcal{R}_{\infw{x}}(u)$, we have $|w|\leq L|u|$. 
In particular, $\infw{x}$ is linearly recurrent with constant $L$ if every length-$(L+1)n$ factor of $\infw{x}$ contains every element of $\cL_n(\infw{x})$. 
We recall the following result.

\begin{lemma}[{\cite[Theorem 24]{Durand-1999}}] \label{lem:Durand1999LR}
Let $\infw{x} \in \cA^{\NN}$ be an infinite aperiodic word that is linearly recurrent with constant $L$.
\begin{enumerate}
    \item For any nonempty word $u\in \cA^*$, $u^{L+1}$ is not a factor of $\infw{x}$. 
    \item For any nonempty $u\in\cL(\infw{x})$ and any $w\in \mathcal{R}_{\infw{x}}(u)$, we have $\frac{|w|}{L}\leq |u|\le L|x|$. 
    \item For any nonempty $u\in\cL(\infw{x})$, we have $\mathrm{Card}\bigl(\mathcal{R}_{\infw{x}}(u) \bigr) \leq L(L+1)^2$. 
\end{enumerate}
\end{lemma}

Consider an infinite word $\infw{x}\in \cA^\NN$ and a finite word $w \in \cA^*$. Fix a constant $C>0$. We say that $w$ is \emph{$C$-balanced in $\infw{x}$} if, for all factors $u,v$ of $\infw{x}$ of equal length, we have $||u|_w - |v|_w| \le C$. We say that $\infw{x}$ is \emph{$C$-letter-balanced}, or \emph{letter-balanced}, if any letter $a$ is $C$-balanced in $\infw{x}$. We say that $\infw{x}$ is \emph{factor-balanced} if any word $w$ is $C_w$-balanced in $\infw{x}$ for some constant $C_w>0$. Notice that we have highlighted that the balance constant could depend on the considered word $w$. Finally, we say that $\infw{x}$ is \emph{uniformly factor-balanced} if there exists a constant $C>0$ such that each finite word $w$ is $C$-balanced in $\infw{x}$. Of course, uniformly factor-balancedness implies factor-balancedness, which also implies letter-balancedness. We have the following classical result about the link between letter-balancedness and uniform letter frequencies. 

\begin{proposition}[{\cite[Proposition 7]{Ada03}}] \label{prop:letterbal-freq}
    An infinite word $\infw{x}\in\cA^\NN$ is letter-balanced if and only if it has uniform letter frequencies and there exists a constant $B$ such that, for any factor $u$ of $\infw{x}$, we have $| |u|_a-\mu_a|u||\leq B$ for any letter $a\in\cA$, where $\mu_a$ is the frequency of $a$. 
\end{proposition}

One can show, with a similar proof, the following result, about uniform factor frequencies and (uniform) factor-balancedness. 

\begin{proposition}
Let $\infw{x}\in\cA^\NN$ be an infinite word.
    \begin{enumerate}
        \item The word $\infw{x}$ is factor-balanced if and only if it has uniform factor frequencies and for any factor $u$ of $\infw{x}$, there exists a constant $B_u$ such that we have $| |u|_a-\mu_a|u||\leq B_u$ for any letter $a\in\cA$, where $\mu_a$ is the frequency of $a$. 
        \item The word $\infw{x}$ is uniformly factor-balanced if and only if it has uniform factor frequencies and there exists a constant $B$ such that, for any factor $u$ of $\infw{x}$, we have $| |u|_a-\mu_a|u||\leq B$ for any letter $a\in\cA$, where $\mu_a$ is the frequency of $a$. 
    \end{enumerate}
\end{proposition}

\subsection{Substitutions}

Given an alphabet $\cA$, a \emph{substitution} is a map $\tau$ from $\cA$ to $\cA^*$ that is compatible with the concatenation, and usually its extension from $\cA^*$ to $\cA^*$ is also called $\tau$. 
If there exists a letter $a$ such that the first letter of $\tau(a)$ is $a$, $\lim_{n\to+\infty}|\tau^n(a)| = +\infty$, and $|\tau(b)|\ne 0$ for every letter $b$, then the limit $\tau^{\omega}(a):=\lim_{n\rightarrow +\infty}\tau^{n}(a)$ exists and is called the \emph{fixed point of $\tau$ starting by $a$}.
In that case, we say that the produced infinite word is \emph{substitutive}.

A substitution $\tau \colon \cA^* \to \cA^*$ is \emph{primitive} if there exists an integer $k$ such that for all letters $a,b \in \cA$, $a$ is a factor of $\tau^k(b)$. 
If $\infw{x}$ is substitutive, 
then one can easily check that $\infw{x}$ is uniformly recurrent whenever the involved substitution is primitive. 

A set $\cL$ of words over $\cA$ is \emph{factor-closed} if, whenever $u$ is a factor of $v \in \cL$, then $u \in \cL$.
It is \emph{biextensible} if $u \in \cL$ implies that $aub \in \cL$ for some letters $a,b \in \cA$.
For a substitution $\tau \colon \cA^* \to \cA^*$, consider the set $L$ formed by all words $u \in \cA^*$ occurring in some $\tau^n(a)$, where $n \ge 1$ and $a \in \cA$.
Then, $L$ contains a unique inclusion-maximal factor-closed and biextensible set of words, which we call the {\em language generated by $\tau$} and denote by $\cL(\tau)$.
We say that $\cL(\tau)$ satisfies a given property (e.g., recurrent or $P$-power-free) if for every $\infw{x} \in \cA^\NN$ such that $\cL(x) \subseteq \cL(\tau)$, $\infw{x}$ also satisfies this property.

A letter $a \in \cA$ is {\em bounded} for $\tau$ if $\bigl(|\tau^n(a)|\bigr)_{n \ge 0}$ is bounded.
We let $B_\tau$ denote the set of bounded letters for $\tau$, and define $\cC_\tau = \cA \setminus B_\tau$, the set of {\em growing letters} for $\tau$.
The substitution $\tau$ is {\em growing} if $\cC_\tau = \cA$, i.e., every letter is growing.
We also say that $\tau$ is {\em quasi-primitive} if there exists $n \ge 1$ such that $b$ occurs in $\tau^n(a)$ for all $a,b \in \cC_\tau$.
This is equivalent to the notion of \emph{$l$-primitiveness} found in~\cite{Rust_tameness,Shimomura_tameness}.
Note that, for growing substitutions, $\tau$ is primitive if and only if $\tau$ is quasi-primitive.
We say that $\tau$ is {\em quasi-uniform} if there exists $C > 0$ such that 
\[  |\tau^n(a)| \leq C \, |\tau^n(b)|
    \enspace \text{for all $n \ge 1$ and $a,b \in \cA$.}
    \]
Finally, $\tau$ is \emph{tame} if and only if $\cL(\tau) \cap \cB_\tau^*$ is finite. 
This definition is equivalent to the original notion of tameness, as shown in~\cite[Theorem 2.9]{Rust_tameness} and~\cite[Proposition 3.17]{Shimomura_tameness}.

\subsection{Discrepancy and balancedness}
\label{sec: discrepancy}

For an infinite word $\infw{x} \in \cA^\NN$ and $u \in \cL(\infw{x})$ such that its frequency $\mu_\infw{x}(u)$ in $\infw{x}$ exists, we define its \emph{discrepancy} by $D_\infw{x}(u) = \sup_{n\geq 0} D_\infw{x}(u,n)$, where 
\begin{align*}
    D_\infw{x}(u,n) = \sup\bigl\{\big||w|_u - |w| \mu_\infw{x}(u)\big| : w \in \cL_n(\infw{x})\bigr\} \in \RR.
\end{align*}
The \emph{balance function} of $u$ in $\infw{x}$ is defined by $B_\infw{x}(u)=\sup_{n\geq 0}B_\infw{x}(u,n)$, where
\begin{align*}
    B_\infw{x}(u,n) = \sup \big\{ \big| |t|_u - |s|_u \big| : t,s \in \cL_n(\infw{x}) \big\}.
\end{align*} 

The basic relationship between the discrepancy and balance function is given by the following lemma, which adapts~\cite[Proposition~7]{Ada03} from the letter case (expressions of the form $|u|_a - |v|_a$ with $a \in \cA$) to the word case.

\begin{lemma}
\label{lem:balance=>discrepancy}
Let $\infw{x}$ be an infinite word and $u \in \cL(\infw{x})$.
\begin{enumerate}
    \item If $u$ is $C$-balanced in $\infw{x}$, then the frequency $\mu_{\infw{x}}(u)$ of $u$ in $\infw{x}$ exists uniformly, and the discrepancy $D_\infw{x}(u)$ is bounded by $C$.
    \item If the discrepancy $D_\infw{x}(u)$ is bounded by $C$, then $u$ is $2C$-balanced in $\infw{x}$.
\end{enumerate}
\end{lemma}

\begin{proof}
We prove the first item of the statement.
Assume that $u$ is $C$-balanced in $\infw{x}$.
For $z \in \cL(\infw{x})$ of length $n + |u| - 1$, we define $g(z) = |z|_u / n$, and we let $m_n$ (resp., $M_n$) denote the minimum (resp., maximum) value of $g(z)$ for $z \in \cL(\infw{x})$ of length $n + |u| - 1$. First, let us show that
\begin{align} \label{eq:lem:balance=>discrepancy:claim} \limsup_{\ell\to\infty} M_\ell \leq M_n 
\end{align}  for every $n \geq 0$.
For $w \in \cL(\infw{x})$ of length $\ell + |u| - 1$ with $\ell\geq n$, we let $k = \lfloor \ell/n \rfloor$.
For each $i\in\{0,\ldots,k-1\}$, we let $z_i$ denote the factor of $w$ of length $n + |u| - 1$ occurring at positions $i n$ and we let $z_k$ be the suffix of $w$ starting at position $kn$.
Since $|w|_u = \sum_{i=0}^k |z_i|_u$, we get 
\[
g(w) = \frac{|w|_u}{\ell}= \sum_{i=0}^k \frac{|z_i|_u}{\ell}=\frac{n}{\ell} \sum_{i=0}^{k-1} g(z_i) + \frac{|z_k|_u}{\ell}.
\]
Since $|z_k| \leq n$, we get $g(w) \leq \frac{nk}{\ell}M_n+\frac{n}{\ell}$, from which~\eqref{eq:lem:balance=>discrepancy:claim} follows.
Similarly, we can prove that 
\begin{align} \label{eq:lem:balance=>discrepancy:claim2}
m_n \leq \liminf_{\ell\to\infty} m_\ell
\end{align} 
for every $n\geq 0$. 
Now, since $u$ is $C$-balanced in $\infw{x}$, we have that $M_n - m_n \leq \frac{C}{n}$ for every $n \geq 0$.
From both Inequalities~\eqref{eq:lem:balance=>discrepancy:claim} and~\eqref{eq:lem:balance=>discrepancy:claim2}, we deduce that the frequency $\mu_\infw{x}(u)$ of $u$ in $\infw{x}$ exists and is the limit of both sequences $(m_n)_{n\ge 0}$ and $(M_n)_{n\ge 0}$.
By Inequalities~\eqref{eq:lem:balance=>discrepancy:claim} and~\eqref{eq:lem:balance=>discrepancy:claim2} one more time, we have $m_n \leq \mu_\infw{x}(u) \leq M_n$ for all $n \geq 0$, which implies that $D_\infw{x}(u) \leq C$ using the bound $M_n - m_n \leq \frac{C}{n}$.

We turn to the proof of the second item of the statement.
Assume that $D_\infw{x}(u)$ is bounded by $C$.
Then, for $w,w' \in \cL(\infw{x})$ of the same length $\ell$, we have
\[	\big| |w|_u - |w'|_u \big| \le 
	\big| |w|_u  - \mu_\infw{x}(u) \ell \big| + 
	\big| |w'|_u - \mu_\infw{x}(u) \ell \big| \le 2C,	\]
which shows that $u$ is $2C$-balanced in $\infw{x}$.
\end{proof}

\subsection{$\cS$-adic representations}

An efficient way to represent low-complexity words is through limits of compositions of substitutions. 
We recall here the basic definitions, which will be central in what follows.

Let $(\cA_n)_{n \ge 0}$ be a sequence of alphabets, and for each $n\ge 0$, consider a substitution $\tau_n \colon \cA_{n+1}^* \to \cA_n^*$.
Assume that there exists a sequence $(b_n)_{n \ge 0}$ of letters such that $b_n\in \cA_n$ for each $n\ge 0$ and each image $\tau_n(b_{n+1})$ begins with $b_n$.
Following the terminology of Berthé, Karimov, and Vahanwala~\cite{Toghrul}, we call $\btau = (\tau_n,b_n)_{n \ge 0}$ a \emph{congenial sequence}.  
For $n > m \ge 0$, we abbreviate
\[
\tau_{m,n} = \tau_m \circ \tau_{m+1} \circ \tau_{m+2} \circ \dots \circ \tau_{n-1}.\]
Because $\tau_n(b_{n+1})$ begins with $b_n$, it follows that $\tau_{m,n}(b_n)$ is a prefix of $\tau_{m,n+1}(b_{n+1})$.
Hence, for each $m \ge 0$, the limit
\[
   \infw{x^{(m)}} = \lim_{n \to \infty} \tau_{m,n}(b_n)
\]
exists in $\cA_m^\NN \cup \cA_m^*$.  
Moreover, one verifies that
\begin{equation}
   \infw{x^{(m)}} = \tau_{m,n}(\infw{x^{(n)}}) 
   \quad \text{for all $n > m \ge 0$.}
\end{equation}

Our main object of study is $\infw{x} \coloneqq \infw{x^{(0)}}$, which we call the \emph{word generated by $\btau$}.  
In this case, $\btau$ is said to be an \emph{$\cS$-adic expansion} of $\infw{x}$.  
The auxiliary infinite words $\infw{x^{(m)}}$ for $m \ge 1$ play a crucial role in the analysis of $\infw{x}$.

\subsection{Sturmian words}

The prototypical example of $\cS$-adic expansions is the one for Sturmian words obtained by Morse and Hedlund~\cite{mh40}, which we describe now. 
We first introduce Sturmian words, following the presentation of Berthé, Holton, and Zamboni~\cite{BertheHoltonZamboni2006}.

Let $\alpha$ be an irrational number in $(0,1)$. 
Consider the two two-interval exchange transformations $R_\alpha : [-\alpha, 1-\alpha) \to [-\alpha, 1-\alpha)$ and  $\widetilde{R}_\alpha : (-\alpha, 1-\alpha] \to (-\alpha, 1-\alpha]$ respectively defined by 
\[
R_\alpha(z) =
\begin{cases}
z + \alpha, & \text{if } z \in [-\alpha, 1 - 2\alpha); \\[6pt]
z + \alpha - 1, & \text{if } z \in [1 - 2\alpha, 1 - \alpha);
\end{cases}
\]
and
\[
\widetilde{R}_\alpha(z) =
\begin{cases}
z + \alpha, & \text{if } z \in (-\alpha, 1 - 2\alpha]; \\[6pt]
z + \alpha - 1, & \text{if } z \in (1 - 2\alpha, 1 - \alpha].
\end{cases}
\]
Both can be considered as rotations of angle $2\pi \alpha$, since these are conjugate, after identification of points $-\alpha$ and $1-\alpha$, to a circle rotation. 

A \emph{Sturmian word} $\infw{x}=(x_n)_{n\ge 0} \in \{0,1\}^{\mathbb{N}}$ \emph{of slope $\alpha$} is the forward itinerary (with respect to the natural partition) of a point $\rho \in [-\alpha, 1-\alpha]$, called the \emph{intercept} of $\infw{x}$, under the action of one of these transformations, i.e., 
\begin{enumerate}
    \item either for all $n \in \NN$, $x_n = 0$ if and only if $R_\alpha^n(x) \in [-\alpha,1-2\alpha)$, 
    \item or for all $n \in \NN$, $x_n = 0$ if and only if $\widetilde{R}_\alpha^n(x) \in [-\alpha,1-2\alpha)$. 
\end{enumerate}
If $z = 0$, then both definitions above produce the same word $\infw{x}$, which is called, in this case, the \emph{characteristic Sturmian word of slope $\alpha$}. 

We now describe the $\cS$-adic expansion of characteristic Sturmian words. For this, let $\theta_0,\theta_1 \colon \{0,1\}^* \to \{0,1\}^*$ be the substitutions defined by
\[
\theta_0 \colon
\begin{cases}
0 &\mapsto 0,\\
1 &\mapsto 01;
\end{cases}
\quad \text{and} \quad
\theta_1\colon
\begin{cases}
0 &\mapsto 10,\\
1 &\mapsto 1.
\end{cases}
\]
We further define, for $n \ge 2$, $\theta_n = \theta_0$ if $n$ is even, and $\theta_n = \theta_1$ if $n$ is odd. The following result is classical, first proven by Morse and Hedlund~\cite{mh40}.
The version cited here follows Berthé, Holton, and Zamboni~\cite{BertheHoltonZamboni2006}, and is presented using the language introduced in the previous subsection.

\begin{proposition}[\cite{mh40}]
\label{lem: S-adic structure of Sturmian words}
Let $\alpha$ be an irrational number in $(0,1)$ with continued fraction expansion $[0;a_1,a_2,\ldots]$.
For $n \ge 0$, define $b_n = 0$ if $n$ is even and $b_n = 1$ otherwise.
Then, $\btau = (\theta_n^{a_{n+1}}, b_n : n \ge 0)$ is a congenial directive sequence generating the characteristic Sturmian word $\infw{x} \in \{0,1\}^\NN$ of slope $\alpha$, i.e., 
\[  \infw{x} = \lim_{n\to\infty}
    \theta_0^{a_1} \circ \theta_1^{a_2} 
    \circ \dots \circ \theta_n^{a_{n+1}}(b_{n+1}).
    \]
\end{proposition}

\subsection{Arnoux--Rauzy words} \label{subsec:ARwords}


In this section, let $\cA$ be a finite alphabet of size $d\ge 2$. An infinite word $\infw{x}$ over $\cA$ is \emph{Arnoux--Rauzy} if it is recurrent, its factor complexity satisfies $\rho_{\infw{x}}(n)=(d-1)n+1$ for all $n\in\NN$, and it has exactly one left special and one right special factor of length $n$ for all $n\in\NN$.
In this paper, we use their description in terms of $\cS$-adic expansions (for instance, see~\cite{BCS2013}), which generalizes the $\cS$-adic expansion of Sturmian words. 
For each letter $a\in \cA$, we define the \emph{Arnoux--Rauzy substitution} $\sigma_a \colon \cA  \to \cA^*$ mapping $a$ to $a$ and any other letter $b \in \cA \setminus \{a\}$ to $ba$.
We then have the following description of Arnoux--Rauzy words in terms of Arnoux--Rauzy substitutions, resembling~\cref{lem: S-adic structure of Sturmian words} for Sturmian words.

\begin{theorem}[\cite{ArnouxRauzy1991}]
An infinite word $\infw{x}$ over $\cA$ is Arnoux--Rauzy if and only if its set of factors coincides with the set of factors of an infinite word $\infw{y}$ of the form \[ \infw{y} = \lim_{n\to\infty}
    \sigma_{a_0} \circ \sigma_{a_1} 
    \circ \dots \circ \sigma_{a_{n+1}}(a),
\]
where $a$ and $a_n$ are letters for any $n\in \NN$ and such that every letter of $\cA$ occurs infinitely often in the sequence $(a_n)_{n\ge 0}$.
\end{theorem}

For a sequence $(a_n)_{n\ge 0}$ of letters of $\cA$, write $a_0a_1a_2\cdots=b_0^{k_0}b_1^{k_1}b_2^{k_2}\cdots$ such that $b_n\in \cA$, $k_n\geq 1$ and $b_n\neq b_{n+1}$ for all $n\geq 0$ (in other words, we group equal consecutive letters in $(a_n)_{n\ge 0}$).
Then $(k_n)_{n\ge 0}$ is called a sequence of \emph{weak partial quotients}.
Let $(K_n)_{n \ge 0}$ be the increasing sequence of integers satisfying $K_0=0$ and
\[
\{a_{K_n}, a_{K_n+1}, \ldots, a_{K_{n+1}}\}=\cA \text{ and } 
\{a_{K_n}, a_{K_n+1},\ldots, a_{K_{n+1}-1}\}\neq \cA \]
for all $n\geq 0$.
The sequence $(K_{n+1}-K_n)_{n\geq 0}$ of first differences is called the sequence of \emph{strong partial quotients}. Note that this terminology comes from the corresponding so-called \emph{Arnoux--Rauzy continued fraction algorithm} (also see~\cite{ArnouxRauzy1991}).




\section{Generalities on factor-balanced words}
\label{sec: generalities on UFB}

We define a natural analogue of the notion of higher block codes between subshifts adapted to the context of one-sided infinite words. 
A {\em higher $k$-block code} is a map $\pi \colon \cA^k \to \cB$, where $k \ge 1$ and $\cA,\cB$ are alphabets.
We let $\pi$ act on words $w = a_0 a_1 \cdots a_{\ell+k-1}$ of length $\ell + k$ by
\[
\pi(w) = \pi(a_0a_1\cdots a_{k-1}) \, \pi(a_1a_2\cdots a_k) \cdots \pi(a_{\ell} a_{\ell+1} \cdots a_{\ell+k-1}).
\]
Let $\infw{x} = (x_n)_{n\ge 0} \in \cA^\NN$ be an infinite word. 
The image of $\infw{x}$ by $\pi$ is the infinite word $\pi(\infw{x}) \in \cB^\NN$ defined by 
\[  \pi(\infw{x}) = \pi(x_0x_1\cdots x_{k-1}) \, \pi(x_1x_2\cdots x_k) \, \pi(x_2x_3\cdots x_{k+1}) \cdots \]
We say that $\pi$ is {\em invertible on $\infw{x}$} if there exists a higher $\ell$-block code $\rho \colon \cB^\ell \to \cA$ such that $\rho(\pi(\infw{x}))$ equals the $(k+\ell)$th-shift of $\infw{x}$, i.e., $\rho(\pi(\infw{x})) = x_{k+\ell} \, x_{k+\ell+1} \cdots$.

One famous example of a higher $k$-block code is the so-called \emph{length-$k$ sliding block code}, which maps distinct length-$k$ factors to distinct letters. We illustrate it on an example below. Also, every length-$k$ sliding block code is invertible, where roughly, the inverse maps every length-$k$ factor to its first letter.

\begin{example}
We consider the ubiquitous Thue--Morse word $\infw{t}=01101001\cdots$, fixed point of the binary substitution $0\mapsto 01, 1\mapsto 10$, see~\cite[\seqnum{A010060}]{Sloane}.
In~\cref{fig:Thue-Morse and 2-block code}, we have represented its length-$2$ sliding block code $\pi(\infw{t})$.
Observe that, since $\infw{t}$ has four length-$2$ factors, $\pi(\infw{t})$ is an infinite word over four letters. More precisely, the factors $01$, $11$, $10$, and $00$ are respectively coded by $0,1,2,3$, which defines $\pi$. Then, $\pi$ is invertible on $\infw{t}$ with inverse $\rho:0,3\mapsto 0; 1,2 \mapsto 1$.

\begin{figure}[ht]
    \centering
\begin{tikzpicture}
  \def\dx{0.55}

  \foreach \x [count=\i] in {0,1,1,0,1,0,0,1,1,0,0,1,0,1,1,0} {
    \node (b\i) at ({\i*\dx},0) {\x};
  }
  \node at ({17*\dx},0) {$\cdots$};

  \foreach \x [count=\i from 1] in {0,1,2,0,2,3,0,1,2,3,0,2,0,1,2} {
    \node (t\i) at ({(\i+0.5)*\dx},-1) {\x};
    \draw[gray!70, dashed, line width=0.3pt] 
      (t\i.north) -- ($ (b\i.north)!0.5!(b\the\numexpr\i+1\relax.north) $);
  }
  \node at ({16.5*\dx},-1) {$\cdots$};

\draw [->] (10,-1) to [bend right] node [right]{$\rho$} (10,0);
\draw [->] (0,0) to [bend right] node [left]{$\pi$} (0,-1);
\end{tikzpicture}
    \caption{The Thue--Morse word $\infw{t}=01101001\cdots$ and its length-$2$ sliding block code $\pi(\infw{t})=012023\cdots$.}
    \label{fig:Thue-Morse and 2-block code}
\end{figure}
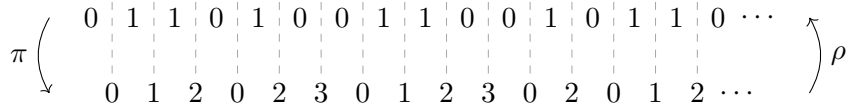
\end{example}

We also provide an example of a non-invertible block code. 

\begin{example}
We consider again the Thue--Morse word $\infw{t}=01101001\cdots$. Let $\pi$ be the $2$-block code defined as $\pi(ab)=a+b\bmod{2}$ for any $a,b\in \{0,1\}$, see~\cref{fig:Thue-Morse and sum block code}.
Let $\infw{p}=\pi(\infw{t})=(p_n)_{n\ge 0}$.
Note that $\infw{p}$ is known under the name \emph{period-doubling}, see~\cite[\seqnum{A035263}]{Sloane}, and is also $2$-automatic (see~\cite{AlloucheShallit2003}).

 \begin{figure}[ht]
    \centering
\begin{tikzpicture}

  \def\dx{0.55}

  \foreach \x [count=\i] in {0,1,1,0,1,0,0,1,1,0,0,1,0,1,1,0} {
    \node (b\i) at ({\i*\dx},0) {\x};
  }
  
  \node at ({17*\dx},0) {$\cdots$};

  \foreach \x [count=\i from 1] in {1,0,1,1,1,0,1,0,1,0,1,1,1,0,1} {
    \node (t\i) at ({(\i+0.5)*\dx},-1) {\x};
    \draw[gray!70, dashed, line width=0.3pt] 
      (t\i.north) -- ($ (b\i.north)!0.5!(b\the\numexpr\i+1\relax.north) $);
  }
  
  \node at ({16.5*\dx},-1) {$\cdots$};

\draw [->] (0,0) to [bend right] node [left]{$\pi$} (0,-1);
\end{tikzpicture}
    \caption{The Thue--Morse word $\infw{t}=01101001\cdots$ and its length-$2$ sum block code. }
    \label{fig:Thue-Morse and sum block code}
\end{figure}

We prove that $\pi$ is not invertible on $\infw{t}$. 
One can easily verify that $\infw{p}$ also satisfies $p_0=1$, $p_1=0$, and the following recurrence relations \[
p_{4n+2}=p_{4n}=p_{2n}, p_{4n+3}=p_n, p_{8n+5}=p_{8n+1}=p_{4n+1}
\]
for all $n\geq 0$. Suppose that $\pi$ is invertible on $\infw{t}$ with inverse $\rho:\{0,1\}^\ell \rightarrow \{0,1\}$ with $\ell\ge 1$ fixed, so $\rho(p_ip_{i+1}\cdots p_{i+\ell-1})=t_{i+\ell+2}$. For all $k\geq 0$, the recurrences relations for $\infw{p}$ give 
\[
p_{2^\ell k}p_{2^\ell k+1}\cdots p_{2^\ell k+\ell-1}=p_0p_1\cdots p_{\ell-1}.
\]
Therefore, this implies that the sequence $(t_{2^\ell k +\ell+2})_{k\ge 0}$ is constant, which is not the case.
\end{example}

This notion adapts that of higher block codes between subshifts to the context of infinite words.
Indeed, there exists a block code $\tilde{\pi}$ mapping a transitive subshift $X$ into $Y$ if and only if there exist a transitive point $\infw{x} \in X$ and a block code $\pi$ such that $\pi(\infw{x}) \in Y$; see~\cite{LindMarcus} for a general background on symbolic dynamics. 
The following simple lemma is left to the reader.

\begin{lemma}
    \label{same_language_same_balance}
    Let $\infw{x},\infw{y}$ be infinite words such that $\cL(\infw{x}) = \cL(\infw{y})$.
    Then, $\infw{x}$ is uniformly factor-balanced if and only if $\infw{y}$ is uniformly factor-balanced.
\end{lemma}

We are now interested by seeing how factor-balancedness is conserved through higher block codes.

\begin{proposition}
\label{balance_inherited_by_block_codes}
Let $k\in\NN$ be a positive integer, let $\infw{x} \in \cA^\NN$ be an infinite word, and let $\pi \colon \cA^k \to \cB$ be a higher $k$-block code. \begin{enumerate}
    \item If $\infw{x}$ is factor-balanced, so is $\pi(\infw{x})$.
    \item If $\infw{x}$ is uniformly factor-balanced and $\pi$ invertible in $\infw{x}$, then $\pi(\infw{x})$ is uniformly factor-balanced.
\end{enumerate}
\end{proposition}
\begin{proof}
Let us prove the first part of the statement.
We show that each factor $w$ in $\pi(\infw{x})$ is $C$-balanced for some constant $C$.
Let $w$ be a factor of $\pi(\infw{x})$.
Since $\infw{x}$ is factor-balanced, for each factor $w'$ of $\infw{x}$, there exists a constant $C_{w'}$ such that  $w'$ is $C_{w'}$-balanced in $\infw{x}$, and define $W = \{w' \in \cL(\infw{x}) \cap \cA^{|w| + k} : \pi(w') = w\}$.
We prove there exists a constant $K > 0$ such that, for all $u,v \in \cL(\pi(\infw{x}))$ with $|u| = |v|$, we have 
\begin{equation}
\label{proof:main_claim:balance_inherited_by_block_codes}
\big| |u|_w - |v|_w \big| \leq  K \cdot \mathrm{Card}(W), 
\end{equation} i.e., $w$ is $(K \, \mathrm{Card}(W))$-balanced in $\infw{\pi(x)}$. 
Set $q = |u| = |v|$.
Since $u,v \in \cL(\pi(\infw{x}))$, there exist $u',v' \in \cL(\infw{x})$ of length $q + k$ such that $\pi(u') = u$ and $\pi(v') = v$.
Then, $\big| |u'|_{w'} - |v'|_{w'} \big| \leq C_{w'}$ for all $w' \in W$.
For each $t\in\{u,v\}$, by the definition of $W$ and since $t = \pi(t')$, we get $|t|_w = \sum_{w' \in W} |t'|_{w'}$.
Therefore,
\[  \big| |u|_{w} - |v|_{w} \big| \leq 
    \sum_{w' \in W} \big| |u'|_{w'} - |v'|_{w'} \big| \leq
    \mathrm{Card}(W) \max_{w' \in W}C_{w'}.  \]
This proves \eqref{proof:main_claim:balance_inherited_by_block_codes} for $K \coloneqq \max_{w' \in W} C_{w'}$, from which it follows that $\pi(\infw{x})$ is factor-balanced.

Next, we prove the second part of the statement, for which it is enough to show that the constants $\mathrm{Card}(W)$ and $K$ appearing in \eqref{proof:main_claim:balance_inherited_by_block_codes} are bounded independently from $w$.
If $\infw{x}$ is furthermore uniformly $C$-factor-balanced, then $K \coloneqq \max_{w' \in W} C_{w'}$ is bounded from above by $C$.
Thus, it is enough to show that $\mathrm{Card}(W)$ is bounded independently from $w$.
Suppose that $\pi$ has an inverse $\rho$, i.e., $\rho$ is an $\ell$-block code such that $\rho(\pi(\infw{x})) = x_{k+\ell} x_{k+\ell+1} \cdots$.
Set $S = \{s \in \cL(\pi(\infw{x})) \cap \cB^{|w| + k + \ell} : \rho(s) \in W\} \cap \cL(\pi(\infw{w}))$.
Then, for any $s \in S$, $\pi(\rho(s)) = w$ is a suffix of $s$.
This shows that $S \subseteq \{t \, w : t \in \cB^{k+\ell}\}$, from which it follows that $\mathrm{Card}(S) \leq \mathrm{Card}(\cB^{k+\ell})$.
As $\rho(S) = W$, we get $\mathrm{Card}(W) \leq \mathrm{Card}(\cB^{k+\ell})$.
We conclude that $K$ is bounded independently of $w$, as desired.
\end{proof}

\begin{remark}
Notice that the first item of the previous proposition fails if we only consider letter-balancedness. For example the Thue--Morse word is letter-balanced while the period doubling word is not. Indeed, define $\tau$ as $\tau(1) = 10$ and $\tau(0) = 11$, then $\infw{x} = \tau^\omega(1)$ is the period doubling word. Define $u_0 = 1$ and $u_{n+1} = \tau^2(u_n) 1$ for $n \ge 0$, i.e., \[   u_n = \tau^{2n}(1) \, \tau^{2(n-1)}(1) \cdots \tau^2(1) \, 1.   \] Let us show that $u_n \in \cL(\infw{x})$. It is clear for $n=0$, and if $u_n \in \cL(\infw{x})$, then $u_n a \in \cL(\infw{x})$ for some $a \in \{0,1\}$, so $\tau^2(u_n a)=\tau^2(u_n)\tau^2(a)$ is in $\cL(\infw{x})$ as well. Since both $\tau^2(0)$ and $\tau^2(1)$ both start with $1$, $u_{n+1}=\tau^2(u_n)1$ is a prefix of $\tau^2(u_n)\tau^2(a)$, and hence it belongs to $\cL(\infw{x})$. 

We now prove that the letter $1$ is not balanced in $\infw{x}$ by means of estimating its discrepancy. The frequency $\mu_1$ of $1$ in $\infw{x}$ is given by the normalized right Perron eigenvector of the incidence matrix of $\tau$, i.e., $\begin{pmatrix}
        0 & 2 \\ 1 & 1
    \end{pmatrix}$; in this case, $\mu_1 = 2/3$. The incidence matrix of $\tau^{2n}$ is \[
\frac{1}{3}
\begin{pmatrix}
2 +4^n & 2(4^n-1) \\ 4^n-1 & 2\cdot 4^n+1
\end{pmatrix}.
\]
Therefore,
\[
|u_n|_1 - \mu_1|u_n| = \sum_{0 \leq i \leq n} \left(|\tau^{2i}(1)|_1 - \mu_1 |\tau^{2i}(1)|\right) =\frac{1}{3} \sum_{0 \leq i \leq n} \left(2\cdot 4^{i}+1-2\cdot 4^{i}\right)=\frac{n+1}{3},
\]
that is, the discrepancy of $1$ is unbounded.
\end{remark}

We now show that factor-balancedness forces recurrence in infinite words to be uniform. 

\begin{proposition}
\label{recurrence&balance=>uniform_recurrence}
Let $\infw{x}$ be a factor-balanced infinite word and let $\infw{y}$ be a recurrent infinite word with $\cL(\infw{y}) \subseteq \cL(\infw{x})$.
\begin{enumerate}
\item The word $\infw{y}$ is uniformly recurrent.
\item For every recurrent infinite word $\infw{y'}$ with $\cL(\infw{y'}) \subseteq \cL(\infw{x})$, we have $\cL(\infw{y}) = \cL(\infw{y'})$.
\end{enumerate}
In particular, every factor-balanced recurrent infinite word is uniformly recurrent.
\end{proposition}

\begin{proof}
Let $\infw{y}, \infw{y'}$ be recurrent infinite words with $\cL(\infw{y}), \cL(\infw{y'}) \subseteq \cL(\infw{x})$, and fix an arbitrary factor $t \in \cL(\infw{y})$.
We claim that 
\begin{align}
\label{eq: claim factor-balanced recurrent}
\text{there exists $\ell \ge 1$ such that every length-$\ell$ factor of $\infw{y'}$ contains $t$.}
\end{align}

Let us first show that Claim~\eqref{eq: claim factor-balanced recurrent} implies the statement.
By Claim~\eqref{eq: claim factor-balanced recurrent}, $t \in \cL(\infw{y'})$ and, moreover, $t$ occurs in $\infw{y'}$ with gaps bounded by $\ell$.
Since $t$ was arbitrary, we obtain $\cL(\infw{y}) \subseteq \cL(\infw{y'})$.
By symmetry (interchanging $\infw{y}$ and $\infw{y'}$), we get $\cL(\infw{y})=\cL(\infw{y'})$, and every word of this common set appears with bounded gaps in both $\infw{y}$ and $\infw{y'}$.
In particular, taking $\infw{y'}=\infw{y}$ yields Item~(1).

It remains to prove Claim~\eqref{eq: claim factor-balanced recurrent}.
As $t\in \cL(\infw{x})$, let $C_t$ be the balance constant of $t$ in $\infw{x}$, i.e., for any two factors $u,v \in \cL(\infw{x})$ of equal length, we have $\big||u|_t - |v|_t\big| \le C_t$.
Because $\infw{y}$ is recurrent and $t \in \cL(\infw{y})$, there exists a factor $p$ of $\infw{y}$ containing at least $C_t+1$ occurrences of $t$.
Set $\ell \coloneqq |p|$.
Now let $w$ be any factor of $\infw{y'}$ with $|w|=\ell$.
Then $p,w \in \cL(\infw{x})$, so by the factor-balancedness of $\infw{x}$,
\[  |w|_t \geq |p|_t - \big| |w|_t - |p|_t \big| \ge |p|_t - C_t. \]
Since $|p|_t \ge C_t+1$, we obtain $|w|_t \ge 1$.
Thus $t$ occurs in every length-$\ell$ factor of $\infw{y'}$, which shows Claim~\eqref{eq: claim factor-balanced recurrent}.
\end{proof}

We now provide an example of a word that is uniformly recurrent but not balanced.

\begin{example}
    Consider the so-called \emph{Chacon word} $\infw{c}=0012001212012\cdots$, which is the fixed point of the primitive substitution $0\mapsto 0012, 1\mapsto 12, 2\mapsto 012$; see~\cite[\seqnum{A049321}]{Sloane}. It is known that $\infw{c}$ is aperiodic, uniformly recurrent, and has linear factor complexity, see~\cite{Ferenczi1995}, but for any word $w \in \{0,1,2\}^{*}$, $w$ is not balanced in $\infw{c}$.
    Indeed, it is known that $\infw{c}$ generates a {\em weakly mixing} topological dynamical system~\cite{Chacon1969}, which are known to admit no balanced factors~\cite{BertheCecchi}.
\end{example}

In combinatorics on words, we say that a nonempty word $w\in\cA^*$ \emph{overlaps itself} if $ww$ contains more than two occurrences of $w$.
Otherwise we say that $w$ \emph{does not overlap itself}; note that then $ww$ contains exactly two occurrences of $w$, one as prefix and one as suffix.

\begin{lemma}[{\cite[Chapter 1]{Lothaire2002}}]
\label{lem:primitive_nonoverlap}
A word is primitive if and only if it does not overlap itself.
\end{lemma}

\begin{lemma}
	\label{factorBal-powerFree:technical}
Let $\infw{x}$ be an infinite word and $u \in \cL(\infw{x})$ be nonempty.
If $u$ is $C$-balanced in $\infw{x}$ for some $C>0$ and $u^{2C+3} \in \cL(\infw{x})$, then $\infw{x}$ is eventually periodic of period dividing $|u|$.
\end{lemma}
\begin{proof}
Without loss of generality, we may assume that $u$ is primitive.
We argue by contradiction and we assume that $\infw{x}$ is not eventually periodic of period dividing $|u|$.
Then, since $u^{2C+3} \in \cL(\infw{x})$, there exists a word $w \in \cL(\infw{x})$ of length $|u^{2C+3}|$ having $u^{C+1}$ as a prefix, but not $u^{C+2}$.
We decompose $w = u^{C+1} t s$, with $|t| = |u|$ and $|s| = (C+1)|u|$; note that $t \neq u$.
Now, $u^{2C+3}$ contains at least $C+2$ occurrences of $u^{C+2}$, so, since $u$ is $C$-balanced in $\infw{x}$, every factor of length $|u^{2C+3}|$ of $\infw{x}$ has at least two occurrences of $u^{C+2}$.
This applies in particular to $w$; let $i < j$ be two different positions at which $u^{C+2}$ occurs in $w$.

On the one hand, we notice that $i \not\in \{ 0,|u|,\dots,C|u|\}$ since, otherwise, $u^{C+2}$ would be a prefix of $w$, contradicting the choice of $w$.
This implies that if $i \leq C|u|$, then $u$ overlaps with itself, which by~\cref{lem:primitive_nonoverlap} is impossible since $u$ is primitive.
Therefore, $i > C|u|$.
On the other hand, we trivially have that $j \leq |w| - |u^{C+2}| = |s| = (C+1)|u|$.
We conclude that $C|u| < i < j \leq (C+1)|u|$.
Thus, the intervals $[i,i+|u^{C+2}|)$ and $[j,j+|u^{C+2}|)$ in $w$ at which $u^{C+2}$ occurs overlap in a window of size at least $|u|$, producing an overlap of $u$ with itself.
In view of~\cref{lem:primitive_nonoverlap}, this is a contradiction.
\end{proof}

Therefore, from the previous lemma, we prove the following two results about power avoidance.

\begin{corollary} \label{factorBal-powerFree:main_prop}
An aperiodic recurrent word that is uniformly factor-balanced  with constant $C$ is $(2C+3)$-power-free.
\end{corollary}
\begin{proof}
If $\infw{x}$ is not $(2C+3)$-power-free, then it is eventually periodic by~\cref{factorBal-powerFree:technical}.
Since $\infw{x}$ is recurrent, this implies that it is periodic, contradicting the hypothesis.
\end{proof}

\begin{lemma}
\label{lem:bound_occs_by_powerfree}
Let $\infw{x} \in \cA^\NN$ be $P$-power-free for an integer $P\ge 1$.
Then, $|v|_u \le \frac{P|v|}{|u|}$ for any $u,v \in \cL(\infw{x})$.
\end{lemma}
\begin{proof}
Suppose, towards a contradiction, that $|v|_u > \frac{P|v|}{|u|}$ for some nonempty $u,v \in \cL(\infw{x})$.
By the Pigeonhole principle, there exist two occurrences of $u$ starting at positions $i$ and $i + \ell$ of $v$ such that $\ell \le |u|/P$.
These two occurrences overlap, which forces the length-$\ell$ prefix $t$ of $u$ to repeat at least $P+1$ times.
In other words, $t^{P+1}$ occurs at position $i$ of $v$, contradicting that $\infw{x}$ is $P$-power-free.
\end{proof}



\section{Factor-balancedness in linearly recurrent words}
\label{sec:UFB&LR}

The objective of this section is to characterize which linearly recurrent words are uniformly factor-balanced.
Our approach relies on an $\cS$-adic description of these words, due to Durand~\cite{Durand-1999} and introduced in~\cref{subsec:LR_as_Sadic}.
The specific substitutive structure of linearly recurrent words allows us to reduce the problem of studying factor-balancedness to that of letter-balancedness across the different levels of the structure.
The resulting characterization is given in~\cref{thm:char_balance_LR}, whose proof is presented in~\cref{subsec:charact_LR_balance}.
As an immediate corollary, we recover Adamczewski’s analogous result for substitutive words~\cite{Ada03,Ada04}, namely~\cref{pro: Adamczewski for primitive substitutive words}.

\subsection{Substitutive description of linearly recurrent words}
\label{subsec:LR_as_Sadic}

We recall the reader that, for a congenial sequence $\btau = (\tau_n \colon \cA_{n+1}^* \to \cA_n^*,\, a_n \in \cA_n)_{n \ge 0}$, the notation $\tau_{m,n}$, for $n > m \ge 0$, stands for the composition $\tau_m \circ \tau_{m+1} \circ \dots \circ \tau_{n-1}$.
Furthermore, we say that $\tau_{m,n}$ is {\em positive} if $b$ occurs in $\tau_{m,n}(a)$ for all $a \in \cA_n$ and $b \in \cA_m$.

The following classical theorem gives a concrete characterization of infinite words that are linearly recurrent in terms of a substitutive structure.
Recall that a substitution $\tau \colon \cA^* \to \cA^*$ is {\em left-proper} if there exists a letter $a$ such that each $\tau(b)$ begins with $a$.

\begin{theorem}[{\cite{Durand-1999}}]
    \label{thm:Sadic_char_LR}
    An infinite word $\infw{x}$ is linearly recurrent if and only if there exists a congenial sequence $\btau = (\tau_n,\, a_n)_{n \ge 0}$ generating $\infw{x}$ satisfying:
    \begin{enumerate}
    \symitem{$(\mathbf{P})$}{defi:L1}
    There exists $k \in \NN$ such that $\tau_{n,n+k}$ is positive for all $n \ge 0$;
    \symitem{$(\mathbf{F})$}{defi:L2}
    The set $\{\tau_n : n \ge 0\}$ is finite; and
   \symitem{$(\mathbf{L})$}{defi:L3}
   The substitution $\tau_n$ is left-proper for all $n \ge 0$.
\end{enumerate}
\end{theorem}

The prototypical example of a congenial sequence $\btau$ as in~\cref{thm:Sadic_char_LR} corresponds to substitutive ones.
More precisely, suppose that $\tau \colon \cA^* \to \cA^*$ is a left-proper primitive substitution and $a \in \cA$.
Then, the constant sequence $\btau \coloneqq (\tau,\, a)_{n \ge 0}$ is congenial and satisfies the three conditions of~\cref{thm:Sadic_char_LR}.
This shows that any fixed point of a left-proper primitive substitution is linearly recurrent.
In fact, any fixed point of any primitive substitution (not necessarily proper) is linearly recurrent~\cite{Durand-1999}.
Therefore, linearly recurrent words generalize primitive substitutive words.
This generalization is vast, since there are only countably many substitutive words but uncountably many linearly recurrent words.
For instance, among Sturmian words, the substitutive ones correspond to slopes that are quadratic algebraic numbers, whereas the linearly recurrent ones correspond to slopes whose continued fraction coefficients are bounded (an uncountable set of zero Lebesgue measure), see~\cref{Sturmians:LR_iff_bounded_coeffs}.

We now collect several facts about linearly recurrent words and the substitutive structure in~\cref{thm:Sadic_char_LR}. Recall that, for a substitution $\tau \colon \cA^* \to \cB^*$, we use the notation $|\tau| = \max_{a\in\cA} |\tau(a)|$. 

\begin{proposition}[\cite{Durand_char_subst_words,Durand-1999}]
\label{prop:Sadic_props_for_LR}
Let $\infw{x}$ be a linearly recurrent word generated by the congenial sequence $\btau = (\tau_n,\, a_n)_{n \ge 0}$ and satisfying Conditions \ref{defi:L1} and \ref{defi:L2} of~\cref{thm:Sadic_char_LR}.
Then, there exists $K > 0 $ such that:
\begin{enumerate}
    \item We have $|\tau_n| \leq K$ and $|\cA_n| \leq K$ for every $n \geq 0$; and 
    \item We have $|\tau_{0,n}(a)| \leq K \cdot |\tau_{0,n}(b)|$ for all $a,b \in \cA_n$ and $n \ge 1$.
\end{enumerate}
\end{proposition}
\begin{proof}
The set $\{\tau_n : n \ge 0\}$ is finite by hypothesis, so there exists $K_1 > 0$ satisfying Item~(1).  
Moreover, Lemma~8 in~\cite{Durand-LR} provides a constant $K_2$ for which Item~(2) holds. 
Thus, $K = \max(K_1, K_2)$ satisfies the statement of the proposition. 
\end{proof}

\begin{lemma} \label{lem:LR_freqs_bound}
Let $\infw{x}$ be an aperiodic linearly recurrent word of constant $L$.
Then:
\begin{enumerate}
    \item The word $\infw{x}$ is $(L+1)$ power-free.
    \item For any $u,v \in \cL(\infw{x})$, $|v|_u \leq (L+1)\frac{|v|}{|u|}$.
    \item The frequency of every $u \in \cL(\infw{x})$ exists and satisfies $\mu_\infw{x}(u) \le \frac{L}{|u|}$.
\end{enumerate}
\end{lemma}
\begin{proof}
Item (1) is contained in Theorem~24 of~\cite{Durand-1999}, and Item (3) in Proposition~13 of~\cite{Durand-LR}.  
Item (2) follows from Item (1) and~\cref{lem:bound_occs_by_powerfree}.
\end{proof}

\subsection{Decisiveness}
\label{sec:decisiveness}

We now introduce a concept that plays a crucial role in the proof of~\cref{thm:char_balance_LR} along with the properties that will be used in its proof.

\begin{definition}
\label{def: decisive substitutions}
Let $\infw{x} \in \cA^\NN$ be an infinite word and $k \in \NN$.
A substitution $\tau \colon \cA^* \to \cB^*$ is {\em $k$-decisive in $\infw{x}$} if there exists a map $r \colon \cA \to \cB^k$ such that $\tau(b)$ starts with $r(a)$ for all $a,b \in \cA$ with $ab \in \cL(\infw{x})$.
For $u \in \cB^*$ of length $|u| \le k + 1$ and $a\in\cA$, we define the notation $q_{\tau,\infw{x}}(u,a)$ as follows.
For $a \in \cA$, we let $r_{|u|}(a)$ denote the prefix of length $|u|-1$ of $r(a)$, and then set
\begin{equation}
	\label{def: decisive substitutions: vector q}
	q_{\tau,\infw{x}}(u,a) = |\tau(a) \, r_{|u|}(a)|_u.
\end{equation}
\end{definition}

We can rephrase decisiveness as follows: it is the property that if a letter $a$ appears in $\infw{x}$ at some position $i$, then this determines not only that the block $\tau(a)$ appears in $\tau(\infw{x})$ at the expected position $j \coloneqq |\tau(\infw{x}_{[0,i)})|$, but also that the extended block $\tau(a) r(a)$ that appears at position $j$.

Decisiveness can be seen a generalization of properness, as described next. Recall that a substitution $\tau \colon \cA^* \to \cA^*$ is left-proper if there exists a letter $a$ such that each $\tau(b)$, $b\in\cA$, begins with $a$. Note that in this case, $\tau$ has a unique fixed point, namely, $\tau^\omega(a)$.

\begin{proposition}
  Let $\tau \colon \cA^* \to \cB^*$ be a left-proper substitution, and  in $\infw{x}$ its fixed point. Then $\tau$ is $1$-decisive in $\infw{x}$. 
\end{proposition}

\begin{proof}
    By setting $r(b) = a$ for $b \in \cA$, this shows that a left-proper substitution is always $1$-decisive in its unique fixed point.
\end{proof}

More generally, for any substitution $\tau \colon \cA^* \to \cA^*$, its length-2 sliding block code is $k$-decisive in any of its fixed points, with $k = \min\{|\tau(a)| : a \in \cA\} - 1$. Furthermore, if a substitution $\tau$ is $1$-decisive in a certain $\infw{x}$, then $\tau^n$ is $k_n$-decisive in $\infw{x}$, where $k_n = \min\{|\sigma^n(a)| : a \in \cA\}$.

One important property of decisive substitutions is that they allow us to compute the occurrences of a given short word $u$ in $\tau(\infw{x})$ using only the occurrences of letters in $\infw{x}$ and some data that depend on $u$, but not on the position where $u$ occurs. 
The next lemma makes this precise.

\begin{proposition}
\label{prop:formula_occurrences_decisive}
Let $\tau \colon \cA^* \to \cB^*$ be a $k$-decisive substitution in $\infw{x} \in \cA^\NN$.
Suppose that $w \in \cA^*$ occurs at position $i \in \NN$ in $\infw{x}$, and let $p = |\tau(\infw{x}_{[0,i + |w|)})|$ be the natural position at which $\tau(w)$ ends in $\tau(\infw{x})$. 
Then, for every $u \in \cL(\tau(\infw{x}))$ of length $|u| \le k + 1$, we have
\begin{equation}
    \label{eq:formula_occurrences_decisive}
    |\tau(w) \, \tau(\infw{x})_{[p, p + |u|)}|_u = 
    \sum_{a\in\cA} |w|_a \, q_{\tau,\infw{x}}(u,a).
\end{equation}
\end{proposition}

\begin{proof}
We need to count occurrences of $u$ whose starting position lies in the interval $[p-|\tau(w)|, p + |u|)$ of $\tau(\infw{x})$. For each $j \in \{i,i+1,\dots,i+|w|-1\}$, consider $\tau(x_j x_{j+1}) = \tau(x_j) \tau(x_{j+1})$.  
Since $\tau$ is $k$-decisive and $|u| \le k+1$, the first $|u|-1$ letters of $\tau(x_{j+1})$ are the fixed word $r_{|u|}(x_j)$, independent of $x_{j+1}$.  
Hence the number of occurrences of $u$ that start in $\tau(x_j)$ (possibly crossing into $\tau(x_{j+1})$) equals
\[
\bigl|\tau(x_j)\, r_{|u|}(x_j)\bigr|_u = q_{\tau,\infw{x}}(u,x_j).
\]
These sets of starting positions when $j$ varies in $\{i,i+1,\ldots,i+|w|-1\}$ are disjoint and together cover exactly the starts of $u$ in the interval $[p-|\tau(w)|, p + |u|)$.
Therefore,
\[
|\tau(w) \, \tau(\infw{x})_{[p, p + |u|)}|_u = 
\sum_{i \leq j < i + |w|} q_{\tau,\infw{x}}(u,x_j) =
\sum_{a \in \cA} |w|_a \, q_{\tau,\infw{x}}(u,a)
\]
since $w = x_i x_{i+1} \cdots x_{i+|w|-1}$.
\end{proof}

We record here the following simple observation, which will be used later.

\begin{lemma}
    \label{lem:composition_decisive}:
    Let $\tau \colon \cA^* \to \cB^*$ and $\sigma \colon \cB^* \to \mathcal{C}^*$ be substitutions.
    Suppose that $\tau$ is $k$-decisive in a word $\infw{x} \in \cA^\NN$.
    Then, $\sigma \circ \tau$ is $\ell$-decisive in $\infw{x}$, where $\ell = k \cdot \min\{|\sigma(b)| : b \in \cA\}$.
\end{lemma}

Recall that, for a substitution $\tau \colon \cA^* \to \cB^*$, we use the notation $\|\tau\| = \sum_{a\in\cA} |\tau(a)|$. 

\begin{proposition}
\label{lem:bound_freqs_convergence}
Let $\infw{x} \in \cA^\NN$ be $B$-letter-balanced, with letter frequencies $\nu_a$ for $a \in \cA$.
Suppose that $\tau \colon \cA^* \to \cB^*$ is $k$-decisive in $\infw{x}$, such that $\tau(\infw{x})$ is $P$-power-free.
Fix $u \in \cL(\tau(\infw{x}))$ of length $|u| \leq k + 1$ and define, using the notation defined in \eqref{def: decisive substitutions: vector q},
\begin{equation}
\label{eq:def_freq_u:lem:bound_freqs_convergence}
	\mu_u = \frac{\sum_{a\in\cA} \nu_a \, q_{\tau,\infw{x}}(u,a)}{\sum_{a\in\cA} \nu_a \, |\tau(a)|}.
\end{equation}
Then, for any $w \in \cL(\infw{x})$,
\begin{equation}
	\label{eq:statement:lem:bound_freqs_convergence}
	\bigl|	|\tau(w)|_u - \mu_u |\tau(w)| \bigr| 
	\le \frac{2PB\|\tau\|}{|u|} + 2P.
\end{equation}
In particular, the frequency of $u$ in $\tau(\infw{x})$ exists and equals $\mu_u$.
\end{proposition}

\begin{proof}
Let $q_{\tau,\infw{x}}(u,w) \coloneqq \sum_{a\in\cA} |w|_a \, q_{\tau,\infw{x}}(u,a)$ be the quantity appearing in \eqref{eq:formula_occurrences_decisive}. 
By~\cref{prop:formula_occurrences_decisive}, $q_{\tau,\infw{x}}(u,w) = |\tau(w) v|_u$ for certain word $v$ of length $|u|-1$.
Using~\cref{lem:bound_occs_by_powerfree}, we can bound $|q_{\tau,\infw{x}}(u,w) - |\tau(w)|_u| \le 2P$.
Therefore, it is enough to prove that $\big| q_{\tau,\infw{x}}(u,w) - \mu_u |\tau(w)|\big|$ is also bounded above by $2PB\|\tau\|/|u|$.

For $a \in \cA$, let $\varepsilon_a = |w|_a - \nu_a |w|$, and note that $|\varepsilon_a| \le B$ by~\cref{lem:balance=>discrepancy}.
We can compute 
\begin{equation}
    \label{eq:1:lem:bound_freqs_convergence}
    q_{\tau,\infw{x}}(u,w) =
	\sum_{a \in \cA} |w|_a \, q_{\tau,\infw{x}}(u,a) = 
	|w|\sum_{a\in\cA} \nu_a\,q_{\tau,\infw{x}}(u,a) +
	\sum_{a\in\cA} \varepsilon_a\,q_{\tau,\infw{x}}(u,a).
\end{equation}
The second term in the right-hand side is small:
From~\cref{lem:bound_occs_by_powerfree} we have $q_{\tau,\infw{x}}(u,a) \le P|\tau(a)|/|u|$, so, since $|\varepsilon_a| \leq B$,
\begin{equation}
    \label{eq:2:lem:bound_freqs_convergence}
    \sum_{a\in\cA} \varepsilon_a\,q_{\tau,\infw{x}}(u,a) \le
	B \cdot \frac{P\|\tau\|}{|u|}.
\end{equation}
Next, we show that the first term in the right side of \eqref{eq:1:lem:bound_freqs_convergence} is close to $\mu_u\, |\tau(w)|$.
Using the definition of $\varepsilon_a$, we can write
\[	|\tau(w)| = 
	\sum_{a\in\cA} |w|_a \, |\tau(a)| =
	|w| \sum_{a\in\cA} \nu_a \, |\tau(a)| +
	\sum_{a\in\cA} \varepsilon_a\,|\tau(a)|.	\]
From this and the definition of $\mu_u$ in \eqref{eq:def_freq_u:lem:bound_freqs_convergence}, we obtain
\[	\Big|
	|w|\sum_{a\in\cA} \nu_a\,q_{\tau,\infw{x}}(u,a)
    - \mu_u|\tau(w)|
	\Big| = 
	\mu_u \Big|\sum_{a\in\cA} \varepsilon_a \, |\tau(a)| \Big|.	\]
Note that $q_{\tau,\infw{x}}(u,a) \leq P|\tau(a)|/|u|$ by~\cref{lem:bound_occs_by_powerfree}, so $\mu_u \le P/|u|$.
This and $|\varepsilon_a| \leq B$ give the estimate:
\begin{equation}
    \label{eq:3:lem:bound_freqs_convergence}
    \Bigl|
	|w|\sum_{a\in\cA} \nu_a\,q_{\tau,\infw{x}}(u,a) -
	\mu_u|\tau(w)| \Bigr| \leq 
	\frac{P}{|u|} \cdot B \cdot \|\tau\|.
\end{equation}
Finally, we can bound $\big| q_{\tau,\infw{x}}(u,w) - \mu_u|\tau(w)| \big|$ by first replacing $q_{\tau,\infw{x}}(u,w)$ using \eqref{eq:1:lem:bound_freqs_convergence} and then estimating the remaining terms with \eqref{eq:2:lem:bound_freqs_convergence} and \eqref{eq:3:lem:bound_freqs_convergence}, i.e.,
\begin{equation*}
    \big| q_{\tau,\infw{x}}(u,w) - \mu_u|\tau(w)| \big| = 
    \Big| |w|\sum_{a\in\cA} \nu_a q_{\tau,\infw{x}}(u,a) - \mu_u|\tau(w)| +
    \sum_{a\in\cA} \varepsilon_a q_{\tau,\infw{x}}(u,a) \Big| 
    \leq \frac{2BP\|\tau\|}{|u|},
\end{equation*}
as desired.
\end{proof}

\subsection{Uniformly factor-balancedness for linearly recurrent words}
\label{subsec:charact_LR_balance}

The main result of this section is~\cref{thm:char_balance_LR}, which we state again for clarity and where we provide conditions under which a linearly recurrent word is (uniformly) factor-balanced.
The criterion is based on a congenial sequence generating the word that satisfies conditions as those in~\cref{thm:Sadic_char_LR}.
It is however convenient to consider the following weakening of Condition \ref{defi:L3}.
This is crucial for treating the substitutive case later.
Let $\btau = (\tau_n,\,a_n)_{n \ge 0}$ be a congenial sequence, and let $\infw{x^{(n)}}$ denote the limit of $(\tau_{n,m}(a_m) )_{m > n}$.
Then, $\btau$ satisfies \ref{defi:L3bis} if
\begin{enumerate}
    \symitem{$(\mathbf{D})$}{defi:L3bis}
    the substitution $\tau_n$ is decisive in $\infw{x^{(n+1)}}$ for every $n \ge 0$.
\end{enumerate}

\maintwo*

Before giving the proof of \cref{thm:char_balance_LR}, we note that a particular case of this result is Adamczewski's result, namely~\cref{pro: Adamczewski for primitive substitutive words} stated in~\cref{sec: our contributions}.

\begin{proof}[Proof of~\cref{pro: Adamczewski for primitive substitutive words}]
As the necessary condition is trivial, we only prove the sufficient condition.
Assume that every length-$2$ factor of $\infw{x}$ is letter-balanced.
Let $\tau \colon \cA^* \to \cA^*$ be a primitive substitution such that $\infw{x}=(x_n)_{n\ge 0}$ is one of its fixed points.
Consider the length-$2$ sliding block code $\tau^{[2]}$ of $\tau$, set $\bar{a} = [x_0 x_1] \in \cA^{[2]}$, and let $\pi \colon \cA^{[2]} \to \cA$ be the map that decodes the $2$-block code back into letters of $\cA$.  
Then, $\tau^{[2]}$ is prolongable on $\bar{a}$, and the generated fixed point $\bar{\infw{x}}$ satisfies $\pi(\bar{\infw{x}}) = \infw{x}$. 
The assumption that every length-$2$ factor of $\infw{x}$ is balanced implies that $\bar{\infw{x}}$ is $C$-letter-balanced for some $C$.

We now define $\tau_n = \tau^{[2]}$ and $a_n = \bar{a}$ for all $n \geq 0$, and let $\btau = (\tau_n,\,a_n)_{n \ge 0}$.  
Then, $\btau$ is congenial and generates $\bar{\infw{x}}$. 
Let us verify that $\btau$ satisfies the hypotheses of~\cref{thm:char_balance_LR}.
First, Conditions \ref{defi:L1} and \ref{defi:L2} of~\cref{thm:Sadic_char_LR} both hold, since $\btau$ is constant (each $\tau_n$ equals the primitive substitution $\tau^{[2]}$).
For \ref{defi:L3bis} now, note that since $\btau$ is constant, all limit points $\bar{\infw{x}}^{\infw{(n)}}$  of $\btau$ coincide with $\bar{\infw{x}}$.
So, as sliding block codes are decisive, each $\tau_{n} = \tau^{[2]}$ is decisive in $\bar{\infw{x}}^{\infw{(n+1)}} = \bar{\infw{x}}$.
Thus, as $\bar{\infw{x}}^{\infw{(n)}}=\bar{\infw{x}}$ for all $n\ge 0$, each $\bar{\infw{x}}^{\infw{(n)}}$ is thus $C$-letter-balanced.

As shown in the previous paragraph, the hypotheses of~\cref{thm:char_balance_LR} are satisfied, and hence $\bar{\infw{x}}$ is uniformly factor-balanced.  
Finally, by~\cref{balance_inherited_by_block_codes} and the fact that $\pi$ is an invertible block code, we conclude that $\infw{x}$ is uniformly factor-balanced as well.
\end{proof}

We now present the proof of~\cref{thm:char_balance_LR}.
We begin with a technical lemma stating that, in substitutive structures such as those in~\cref{thm:char_balance_LR}, for any word $u$ one can find a suitable level $n$ of the structure whose scale $|\tau_{0,n}|$ is comparable to the length of the word $u$.
This is an essential step in efficiently controlling the bound given by~\cref{lem:bound_freqs_convergence}.

\begin{lemma}
    \label{lem:find_representative}
    Let $\boldsymbol{\tau} = (\tau_n \colon \cA_{n+1}^* \to \cA_n,\, a_n \in \cA_n)_{n \ge 0}$ be a congenial sequence satisfying Items (1) and (2) of~\cref{prop:Sadic_props_for_LR} with constant $K > 0$, as well as Condition \ref{defi:L3bis}.
    Let $\infw{x^{(m)}}$ denote the limit of $(\tau_{m,n},a_n)_{n > m}$.
    Then, for every $u \in \cL(\infw{x^{(0)}})$, there exists $n \ge 1$ such that
    \begin{enumerate}[label=(\roman*)]
        \item The substitution $\tau_{0,n}$ is $|u|$-decisive in $\infw{x^{(n)}}$; and
        \item We have $|\tau_{0,n}(a)| \le K^3|u|$ for all $a \in \cA_n$.
    \end{enumerate}
    In particular, $|u|\leq |\tau_{0,n}(a)| \leq K^3|u|$ for all $a \in \cA_n$.
\end{lemma}
\begin{proof}
    Let $u \in \cL(\infw{x^{(0)}})$.
    Since $|\tau_{0,n}(a_n)| \to \infty$ as $n \to \infty$, there exists a largest $n \ge 1$ such that $|\tau_{0,n}(a)| \le K^3|u|$ for all $a \in \cA_n$.
    Then, $n$ satisfies Item (ii).

    We show that $n$ also satisfies Item (i).
    By Condition \ref{defi:L3bis}, the fact that $\tau_{m-1}$ is decisive in $\infw{x}^{(m)}$ for every $m \ge 1$ implies, by~\cref{lem:composition_decisive}, that $\tau_{0,n} = \tau_{0,n-1} \circ \tau_{n-1}$ is $\ell_n$-decisive in $\infw{x^{(n)}}$, where $\ell_n = \min\{|\tau_{0,n-1}(a)| : a \in \cA_{n-1}\}$.
    So, to obtain Item (i), it suffices to prove that $|u| \le \ell_n$.
    Using the maximality of $n$, we can find $b \in \cA_{n+1}$ such that $K^3|u| < |\tau_{0,n+1}(b)|$, so we can estimate
    \[  K^3|u| < |\tau_{0,n+1}(b)| 
        \leq |\tau_{n-1} \circ \tau_n| \cdot |\tau_{0,n-1}|.
        \]
    Item (1) of~\cref{prop:Sadic_props_for_LR} gives $|\tau_{n-1} \circ \tau_{n}| \le |\tau_{n-1}|\cdot |\tau_n|\le K^2$ and Item (2) of~\cref{prop:Sadic_props_for_LR} that $|\tau_{0,n-1}| \le K \ell_n$.
    We conclude that $K^3|u| < K^3 \ell_{n}$, and thus $|u| \le \ell_{n+1}$, as desired.
\end{proof}

\begin{proof}[Proof of~\cref{thm:char_balance_LR}]

We fix a linearly recurrent word $\infw{x}$ with constant $L$, and we let $\btau = (\tau_n \colon \cA_{n+1}^* \to \cA_n^*,\, a_n \in \cA_n)_{n \ge 0}$ be the congenial sequence generating $\infw{x}$ given by~\cref{thm:Sadic_char_LR}.
We further define $K$ to be the constant appearing in~\cref{prop:Sadic_props_for_LR}.

Assume that, for all $n \geq 0$, there exists $B_n > 0$ such that $\infw{x^{(n)}}$ is $B_n$-letter-balanced.
Furthermore, assume that $\infw{x}$ is aperiodic, since otherwise the result is trivial. 
We show that the discrepancy of a given factor $u$ of $\infw{x}$ is bounded by $2(L+1)(B_nK^4+2K^3+1)$, for some $n$ that depends only on $u$.
This implies, by~\cref{lem:balance=>discrepancy}, that $\infw{x}$ is factor-balanced, as desired.

Let $u, w \in \cL(\infw{x})$.
Our goal is to estimate the quantity $\big||w|_u - \mu_\infw{x}(u)|w|\big|$.
We begin by introducing the necessary objects.
The hypotheses allow us to use~\cref{lem:find_representative} with $\btau$ and $u$, which yields an integer $n \ge 1$ such that
\begin{enumerate}[label=(\roman*)]
    \item The substitution $\tau_{0,n}$ is $|u|$-decisive in $\infw{x^{(n)}}$;
    \item We have $|\tau_{0,n}(a)| \le K^3|u|$ for all $a \in \cA_n$.
\end{enumerate}
Following \eqref{eq:def_freq_u:lem:bound_freqs_convergence}, we define
\[  \mu_u = \frac{\sum_{a\in\cA_n} \nu_a\, q_{\tau_{0,n},\infw{x^{(n)}}}(u,a)}{\sum_{a\in\cA_n} \nu_a\, |\tau_{0,n}(a)|},  \]
where $\nu_a$ is the frequency of $a$ in $\infw{x^{(n)}}$, which exists by~\cref{prop:letterbal-freq}, as $\infw{x^{(n)}}$ is letter-balanced.
Since $w$ occurs in $\tau_{0,n}(\infw{x^{(n)}})$, there exists $w' \in \cL(\infw{x^{(n)}})$ such that $w$ appears in $\tau_{0,n}(w')$.
We take the shortest such word.

We now estimate the discrepancy of $u$ in $w$ through the discrepancy of letters of $\tau_{0,n}(w')$.
More precisely, we decompose 
\begin{equation}
    |w|_u - \mu_u|w| =
    \underbrace{|w|_u - |\tau_{0,n}(w')|_u}_{:=T_1}  + 
    \underbrace{|\tau_{0,n}(w')|_u - \mu_u|\tau_{0,n}(w')|}_{:=T_2} +
    \underbrace{\mu_u|\tau_{0,n}(w')| - \mu_u|w|}_{:=T_3},
\end{equation}
and proceed to bound $|T_1|$, $|T_2|$, and $|T_3|$ separately.

We start with some general observations.
Since $w$ occurs in $\tau_{0,n}(w')$, we can write $p \, w \, s = \tau_{0,n}(w')$ for some $p,s \in \cA^*$.
Moreover, the minimality of $w'$ ensures that if $\alpha$ and $\beta$ are the first and last letters of $w'$, then $p$ is a prefix of $\tau_{0,n}(\alpha)$ and $s$ is a suffix of $\tau_{0,n}(\beta)$.
Hence, by Item (ii) above, we have
\begin{equation}
    \label{eq:bound_p_s}
    |p| \leq K^3|u|
    \enspace \text{and} \enspace
    |s| \leq K^3|u|.
\end{equation}
Also, as $\infw{x}$ is aperiodic and linearly recurrent with constant $L$,~\cref{lem:LR_freqs_bound} ensures that 
\begin{equation}
    \label{eq:bounds:|w|_u & mu_u}
    \mu_u \le \frac{L}{|u|}
    \enspace \text{and} \enspace
    |t|_u \le \frac{(L+1)|t|}{|u|}
    \enspace\text{for every $t \in \cL(\infw{x})$.}
\end{equation}

We now estimate the terms $T_1, T_2, T_3$.
As  $\tau_{0,n}(w')=pws$, note that $|T_1|$ corresponds to the number of occurrences of $u$ in $\tau_{0,n}(w')$ that either start within the prefix $p$ or end within the suffix $s$.
Since $|p|, |s| \leq K^3|u|$ by \eqref{eq:bound_p_s}, we can use \eqref{eq:bounds:|w|_u & mu_u} with $t\in\{p,s\}$ to obtain the bound \begin{align} \label{eq:T1}
    |T_1| \le 2(L+1)K^3.
\end{align} 

For $T_3$, we observe that $|\tau_{0,n}(w')| - |w| = |s| + |p|$, so we use~\eqref{eq:bound_p_s} and \eqref{eq:bounds:|w|_u & mu_u} to obtain
\begin{align}\label{eq:T3}
    |T_3| = \mu_u ||\tau_{0,n}(w')| - |w|| = \mu_u\,(|p| + |s|) \leq 2(L+1)K^3.
\end{align}

Finally, to control $|T_2|$, observe that since $\infw{x^{(n)}}$ is $B_n$-letter-balanced and $\tau_{0,n}(\infw{x^{(n)}})=\infw{x}$ is $(L+1)$-power-free by~\cref{lem:LR_freqs_bound}, we can use~\cref{lem:bound_freqs_convergence} to obtain 
\[
|T_2| \le \frac{2(L+1)B_n\|\tau_{0,n}\|}{|u|}+2(L+1).
\]
Since $\|\tau_{0,n}\| \le |\cA_n|\cdot|\tau_{0,n}|$, the bounds $|\cA_n| \le K$ from Item (1) of~\cref{prop:Sadic_props_for_LR} and $|\tau_{0,n}| \le K^3|u|$ from Item (ii) above finally give that \begin{align} \label{eq:T2}
    |T_2| \le 2(L+1)B_n K^4+2(L+1)=2(L+1)(B_nK^4+1).
\end{align}
We conclude by combining~\eqref{eq:T1}, \eqref{eq:T3}, and~\eqref{eq:T2} that 
\begin{align}
\label{eq: final bound}
    \big| |w|_u - \mu_u|w|\big| &\le 
    |T_1| + |T_2| + |T_3|  \nonumber \\
    &\le
    2(L+1)K^3
    + 2(L+1)(B_nK^4+1)
    + 2(L+1)K^3 \nonumber \\ & \leq 2(L+1)(B_nK^4+2K^3+1). 
\end{align}
as desired.

To prove the last part of the statement, if the sequence $(B_n)_{n\ge 0}$ is bounded, then~\eqref{eq: final bound} proves that the word $\infw{x}$ is $C$-uniformly factor bounded, with $C= 2(L+1)(\sup_{n\geq 0}B_n \, K^4+2K^3+1)$.
\end{proof}

One may ask for the optimality of the bound obtained in~\eqref{eq: final bound}.
In the next example, we study more explicitly the case of the Tribonacci word.

\begin{example}
\label{ex: constant for the Tribonacci word}
Define the Tribonacci substitution $\tau \colon \cA^* \to \cA^*$ by $a \mapsto ab$, $b \mapsto ac$, and $c \mapsto a$ and $\infw{t}$ its unique fixed point, called the \emph{Tribonacci word}.
Then, the constant congenial sequence $\btau = (\tau,\,a)_{n \ge 0}$ generates $\infw{t}$.
Moreover, since $\btau$ is constant, all the limit points are $\infw{t^{(n)}}$ generated by $\btau$ are equal to $\infw{t}$.
We establish the following quantities:
\begin{itemize}
    \item Since it is known that $\infw{t}$ is $2$-letter-balanced, see~\cite{RSZ:2010}, and $\infw{t^{(n)}} = \infw{t}$ for all $n \ge 0$, we get $B_n = 2$ for all $n\geq 0$.
    \item All quantities $|\cA|$, $|\tau|$ and $\sup\{ \frac{|\tau^n(a)|}{|\tau^n(b)|} : a,b \in \cA,\, n\ge1\}$ of~\cref{prop:Sadic_props_for_LR} are bounded by $3$, so we may choose $K = 3$.
    \item The linear recurrence constant satisfies $L=2\psi^2 + \psi + 1 \approx 9.605$, where $\psi$ is the Tribonacci root. We refer the reader to~\cref{appendix: Trib} for more details on how to compute this constant (using the theorem prover \texttt{Walnut}). 
\end{itemize}
Thus, the bound in~\eqref{eq: final bound} is equal to $2(2\psi^2+\psi+2)(2\cdot 3^4 + 2\cdot 3^3 + 1) \approx 4602.6734$, i.e., the balance constant is at most $4602$. Note that the optimal constant for the Tribonacci word is $2$, as proven in~\cite[Theorem 22]{CDOPSS2025} (also see~\cref{thm: Tribo is 2 UFB}).
\end{example}

\begin{remark}
One may replace the hypothesis of being linearly recurrent in \cref{thm:char_balance_LR} with that of being power-free and satisfying both Items (1) and (2) of \cref{prop:Sadic_props_for_LR}. Indeed, this hypothesis was only used when applying~\cref{lem:LR_freqs_bound} and~\cref{prop:Sadic_props_for_LR}. However, the same conclusions of~\cref{lem:LR_freqs_bound} follow under power-freeness by \cref{lem:bound_occs_by_powerfree} and the fact that the frequency of a balanced letter always exists, see~\cref{lem:balance=>discrepancy}. Therefore, one can rephrase our theorem as the following.

\begin{theorem} \label{thm:main_thm_weaker_hyp}
Let $\infw{x}$ be a power-free word and $(\tau_n,\, a_n)_{n \ge 0}$ be a congenial sequence generating $\infw{x}$ satisfying Conditions \ref{defi:L1}, \ref{defi:L2}, and \ref{defi:L3bis}. Suppose that both Items (1) and (2) of \cref{prop:Sadic_props_for_LR} hold. 
Let $\infw{x^{(n)}}$ denote the limit of $(\tau_{n,m}(a_m))_{m > n}$.
Suppose that there exists a sequence $(B_n)_{n\geq 0}$ such that the word $\infw{x^{(n)}}$ is $B_n$-letter-balanced. Then the word $\infw{x}$ is factor-balanced. 
Moreover, if the sequence $(B_n)_{n\geq 0}$ is bounded, then the word $\infw{x^{(n)}}$ is uniformly factor-balanced. 
\end{theorem}
    
We will apply this slight generalization to a specific example in the next section, where the considered word will be uniformly factor-balanced but not recurrent. 
\end{remark}


\section{Uniform factor-balancedness in Arnoux--Rauzy words} \label{sec:UFB_AR}

The goal of this section is to prove \cref{thm:mainAR} recalled below.

\mainAR*

We first deal with the case of a two-letter alphabet, i.e., Sturmian words, then the case of ternary Arnoux-Rauzy words.

\subsection{Sturmian words} \label{sec:Sturmian and UFB}

The original source of the next result states it for subshifts instead of infinite words, but since Sturmian words are uniformly recurrent~\cite{mh40}, both statements are equivalent.

\begin{theorem}[{\cite[Proposition 10]{Durand-LR}}] 
\label{Sturmians:LR_iff_bounded_coeffs}
Let $\infw{x}$ be a Sturmian word of slope $\alpha\in(0,1)$.
Then $\infw{x}$ is linearly recurrent if and only if the coefficients of the continued fraction of $\alpha$ are bounded.
\end{theorem}

\begin{proof}[Proof of \cref{thm:mainAR} in the binary case]
The sufficient condition is already proven by Fagnot and Vuillon~\cite[Corollary 13]{FV:2002}, although stated in different terms.

It remains to prove the necessary condition.
As the Sturmian word of irrational slope $\alpha$ and intercept $\rho$ and the characteristic Sturmian word of slope $\alpha$ have the same set of factors (see~\cite[Proposition 2.1.18]{Lothaire2002}), by \cref{same_language_same_balance} we may only care about the case where $\infw{x}$ is this characteristic Sturmian word. Then, by~\cref{lem: S-adic structure of Sturmian words}, $\infw{x}$ admits the $\cS$-adic expansion
\[  \infw{x} = \lim_{n\to\infty}
        \theta_0^{a_1} \circ \theta_1^{a_2} \circ \dots \circ \theta_1^{a_{2n}}(0),
        \]
where $\alpha = [0;a_1,a_2,\dots]$ is the continued fraction of $\alpha$.
For all integers $n$, define the words $u_n = \theta_0^{a_1} \circ \theta_1^{a_2} \circ \dots \circ \theta_{\bar{n}}^{a_{n}}(0)$ and $v_n = \theta_0^{a_1} \circ \theta_1^{a_2} \circ \dots \circ \theta_{\bar{n}}^{a_{n}}(1)$, where $\bar{n}=n \mod 2$. 
Then, for even $n$, $u_{n+1} = u_n^{a_n} v_n$, and for odd $n$, $v_{n+1} = v_n^{a_n} u_n$.
Since $u_{n+1}, v_{n+1} \in \cL(\infw{x})$ and $\infw{x}$ is an aperiodic $C$-uniformly factor-balanced word, we can use~\cref{factorBal-powerFree:main_prop}, yielding $a_n \leq 2C + 3$ for all $n$.
Therefore, the coefficients of the continued fraction of $\alpha$ are uniformly bounded by $2C+3$.
Finally, observe that $(a_n)_{n\ge 1}$ forms the sequence of weak partial quotients, which ends the proof.
\end{proof}

\begin{remark}
    We mention that one may also use techniques similar to those of the next subsection to show the sufficient condition of the precedent result, using \cref{thm:main_thm_weaker_hyp}.
\end{remark}

\cref{Sturmians:LR_iff_bounded_coeffs} implies the following immediate corollary. 

\begin{corollary}
    A Sturmian word $\infw{x}$ is uniformly factor-balanced if and only if $\infw{x}$ is linearly recurrent.
\end{corollary}

\subsection{Arnoux--Rauzy over three letters}

It is known that bounded strong partial quotients characterize Arnoux--Rauzy words that are linearly recurrent. 

\begin{theorem}[{\cite[Proposition 8.4]{BST2019}}]
An Arnoux-Rauzy word over any alphabet is linearly recurrent if and only if it has bounded strong
partial quotients, that is, each Arnoux--Rauzy substitution occurs in its directive sequence with bounded gaps.
\end{theorem}

Notice that the notion of strong and weak partial quotients are the same in the Sturmian words case. 
On the other hand, only a sufficient condition for letter-balancedness of Arnoux--Rauzy in terms of bounded weak partial quotients is possible. Indeed, Berthé et al. \cite{BCS2013} proved the following result. 

\begin{theorem}[{\cite[Theorem 7]{BCS2013}}] \label{thm:AR_letter-balanced}
    Let $\infw{x}$ be an Arnoux-Rauzy word over a $3$-letter alphabet with $\cS$-directive sequence $(i_m)_{m\ge 0}$. If the weak partial quotients are bounded by $C$, then $\infw{x}$ is $(2C+1)$-letter-balanced.
\end{theorem}

However, the converse of the previous result is not true, since authors of~\cite{BCS2013} provide an example of a $2$-letter balanced Arnoux--Rauzy word with unbounded weak partial quotients.


To prove~\cref{thm:mainAR} in the ternary case, we will make use of~\cref{thm:main_thm_weaker_hyp}. Notice that we could also prove the binary case with this method. Before that, we first prove the two following lemmas. Note that they hold for any alphabet size. 

\begin{lemma} \label{lem:images_mainAR}
Let $\btau = (\tau_n)_{n\ge0}$ be a sequence of Arnoux--Rauzy substitutions over $\cA$ with bounded weak partial coefficients.
For all $a,b\in\cA$, the sequence $\bigl(|\tau_{0,n}(a)| / |\tau_{0,n}(b)|\bigr)_{n\ge 0}$ is uniformly bounded.
\end{lemma}

\begin{proof}
Given a substitution $\eta$ on $\cA$, let $T(\eta) = \max_{a,b\in\cA} |\eta(a)| / |\eta(b)|$.
Note that for any Arnoux--Rauzy substitution $\alpha$, we have
\begin{align}\label{eq:singleAR}
    T(\eta\alpha)\le T(\eta)+1.
\end{align}
Furthermore, a direct computation of letter counts gives for any distinct Arnoux--Rauzy substitutions $\alpha,\beta$ and any substitution $\eta$ 
\begin{align} \label{eq:doubleAR}
    T(\eta \alpha \beta)\le \frac{1}{2}T(\eta)+2.
\end{align} 
We first show that this inequality holds in the binary case. suppose that $\alpha = \sigma_a$, $\beta = \sigma_b$ with $a\neq b$.
Observe that $\sigma_a \circ \sigma_b$ maps $a \mapsto aba$, $b \mapsto ba$. Suppose that $\eta(a) \ge \eta(b)$, so $T(\eta)=\frac{\eta(a)}{\eta(b)}$, and then \[ T(\eta \sigma_a \sigma_b) = \frac{|\eta(aba)|}{|\eta(ba)|} = 1+\frac{|\eta(a)|}{|\eta(ba)|} \leq 1+\frac{|\eta(a)|}{2|\eta(b)|}\leq 1+\frac{T(\eta)}{2}.
\] The case where $\eta(b) > \eta(a)$ works in a similar way.  
Now, assume that $|\cA|\ge 3$ and suppose that $\alpha = \sigma_a$, $\beta = \sigma_b$ with $a\neq b$.
Observe that $\sigma_a \circ \sigma_b$ maps $a \mapsto aba$, $b \mapsto ba$, $c \mapsto caba$ for $c \in \cA \setminus \{a,b\}$. Let $c_{m} \in \cA \setminus \{a,b\}$ (resp., $c_{M}$) be such that $|\eta(c_{m})|$ is the shortest (resp., the longest) word in $\bigl\{\eta(c) : c \in \cA \setminus\{a,b\}\bigl\}$. Therefore, the shortest possible image of $\eta \alpha \beta$ is $\eta(ba)$, and the longest one is $\eta(c_Maba)$. 
Then, we have 
\[ T(\eta \sigma_a \sigma_b) = \frac{|\eta(c_Maba)|}{|\eta(ba)|} = 
1 + \frac{|\eta(c_Ma)|}{|\eta(ba)|} \le 
2 + \frac{|\eta(c_M)|}{|\eta(ba)|} \le 
2 + \frac{|\eta(c_M)|}{2|\eta(c_{m})|}  \le
2 + \frac{1}{2}T(\eta).
\]
Therefore, \eqref{eq:doubleAR} is proven for all sizes of $\cA$.

Now, since $\btau$ has bounded weak partial coefficients, there is an increasing sequence $(n_j)_{j \ge 0}$ of positive integers such that, for all $j\ge 0$, $\tau_{n_j} \neq \tau_{n_j+1}$ and such that $2 \le n_{j+1} - n_j \le C$ for some constant $C \ge 0$. By \eqref{eq:doubleAR}, we have
\begin{align*}
    T(\tau_{0,n_j+2}) =
    T(\tau_{0,n_j} \circ \tau_{n_j} \circ \tau_{n_j+1}) \le 
    \frac{1}{2}T(\tau_{0,n_j}) + 2
    \end{align*} 
    for every $j \ge 0$.
Moreover, we have $0 \le n_{j} - (n_{j-1}+2) \le C-2$ for all $j\ge 1$, so by inductively applying~\eqref{eq:singleAR}, the previous inequality becomes
\begin{align*}
    T(\tau_{0,n_j+2}) \le 
    \frac{1}{2}T(\tau_{0,n_j}) + 2 &\le 
    \frac{1}{2}T(\tau_{0,n_{j-1}+2}) + \frac{C-2}{2} +2, \\
    &\leq 
    \frac{1}{2}T(\tau_{0,n_{j-1}+2}) + \frac{C}{2}+1
    \end{align*} 
for every $j \ge 1$.
For all $j\ge 1$, iterating the previous inequality yields the bound
\begin{equation}
\label{eq:relationship between nj+2 and n0+2}
T(\tau_{0,n_j+2}) \le
    \frac{1}{2^j}T(\tau_{0,n_0+2}) + (C+2)    
\end{equation}
as $\sum_{i=0}^{j-1} \frac{1}{2^i} (\frac{C}{2}+1) \le C+2$.
Between the indices $n_{j+1}+2$ and $n_j+2$, there are $n_{j+1}-n_j\le C$ substitutions, so applying \eqref{eq:singleAR} at most $C$ times gives $T(\tau_{0,m}) \le T(\tau_{0,n_j+2})+C$ for all $m \in [n_j+2, n_{j+1}+2)$.
Putting this and~\eqref{eq:relationship between nj+2 and n0+2} together, we conclude that $T(\tau_{0,m}) \le \frac{1}{2^j}T(\tau_{0,n_0+2})+2C+2$, which gives a uniform bound.
\end{proof}

\begin{lemma}
    \label{lem:powerfree_mainAR}
     An Arnoux--Rauzy word over $\cA$ with bounded weak partial coefficients is $P$-power free for some $P \ge 0$.
\end{lemma}

\begin{proof}
Let $\infw{x}$ be an Arnoux--Rauzy word over $\cA$ given by the directive sequence $\btau = (\tau_n)_{n \ge 0}$ with bounded weak partial coefficients.
Let $C_1$ be an upper bound for the weak partial coefficients.
By~\cref{lem:images_mainAR}, there exists a constant $C_2 \ge 0$ such that $|\tau_{0,n}(a)|/|\tau_{0,n}(b)| \le C_2$ for all $a,b \in \cA$ and $n\ge0$.
Set $P = 16C_2^2(C_1+5)$.
We argue by contradiction and assume that $\infw{x}$ is not $P$-power free.
Let $u \in \cL(\infw{x})$ be nonempty with $u^P \in \cL(\infw{x})$. 
We may suppose that $u$ is primitive, i.e., $u$ is not a power of a strictly shorter word. The proof is divided into three steps: the first one consists of finding an integer $n$ and a sufficiently long word $v \in \cL(\infw{x}^{(n+1)})$ such that $\tau_{0,n+1}(v)$ occurs in $u^P$ and its period $|u|$ is much smaller than $|\tau_{0,n}(a)|$ for all $a \in \cA$. Then, the second step consists of proving that this strong periodicity forces $v$ to actually be of the form $v=v_1\lambda^{|v|-1}$ for some letters $v_1,\lambda \in \cA$. Finally, we use the assumption of bounded weak partial quotients to prove that such a $v$ does not exist. 

\textbf{Step 1}:
First, take $n \ge 1$ to be the least integer such that \begin{align} \label{eq:choiceofn}
    |\tau_{0,n}(a)| \ge 4|u|
\end{align} for all $a \in \cA$.
Note that for all $a,a' \in \cA$ and $m \ge 1$, the definition of Arnoux--Rauzy substitutions and that of $C_2$ give
\[  |\tau_{0,m+1}(a)| \le 2 \max_{b \in \cA} |\tau_{0,m}(b)| \le 2C_2 |\tau_{0,m}(a')|.
\]
Applying this inequality twice (where $\tau_{0,m}$ is the identity map if $m = 0$) gives $|\tau_{0,n+1}(a)| \le 2C_2 |\tau_{0,n}(a')| \le 4C_2^2|\tau_{0,n-1}(a')|$ for all $a,a' \in \cA$.
Now, by the minimality of $n$, there exists $a' \in \cA$ such that $|\tau_{0,n-1}(a')| < 4|u|$.
Therefore, 
\begin{equation}
 \label{eq:with 16 and C2 squared}
 |\tau_{0,n+1}(a)| \le 4 C_2^2|\tau_{0,n-1}(a')| < 16 C_2^2|u|
\end{equation}
for all $a \in \cA$. Let $\infw{x}^{(n+1)}$ be the Arnoux--Rauzy word generated by $(\tau_m)_{m>n+1}$.
Since $u^P$ occurs in $\infw{x}$, there exists $v \in \cL(\infw{x}^{(n+1)})$ such that $\tau_{0,n+1}(v)$ occurs in $u^P$. 
We take $v$ to be of maximal length. 
By maximality of $v$, we have $|u^P| \le |\tau_{0,n+1}(v)| + 2\max_{b\in\cA}|\tau_{0,n+1}(b)|$, which in turn gives
\[
P |u| = |u^P| \le |\tau_{0,n+1}(v)| + 2\max_{b\in\cA}|\tau_{0,n+1}(b)| \le (|v|+2)\cdot 16 C_2^2 |u|
\]
by using~\eqref{eq:with 16 and C2 squared}.
Therefore, we have 
\[|v| \ge \frac{P}{16 C_2^2}-2 \ge C_1+3.\]

\textbf{Step 2}: We prove that $v$ ends with $\lambda^{|v|-1}$ for some letter $\lambda \in \cA$.
Notice that Arnoux--Rauzy substitutions are right-proper, i.e., each image $\tau_n(a)$ ends with the same fixed letter; say $a_n$.
Therefore, $\tau_{0,n+1}(a)$ ends with $w_n \coloneqq \tau_{0,n}(a_n)$ for all $a \in \cA$.
Then, by \eqref{eq:choiceofn}, $|w_n| \ge 4 |u|$.
Also, for each nonempty prefix $pa$ of $v$, with $p \in \cA^*$ and $a \in \cA$, there is an occurrence of $\tau_{0,n+1}(a)$ in $\tau_{0,n+1}(v)$ ending at position $\ell(pa) \coloneqq |\tau_{0,n+1}(pa)|$. 
Let us write $v = v_1 v_2 \cdots v_{|v|}$ with $v_i \in \cA$ for all $i\in\{1,\ldots,|v|\}$.
We claim that 
\begin{equation}
 \label{eq: the claim on length mod u}
 |\tau_{0,n+1}(v_i)| \equiv 0 \pmod{|u|}
\end{equation} 
for all $i\in\{2,3,\ldots,|v|\}$.
Indeed, for each such $i$, there is an occurrence of $\tau_{0,n+1}(v_i)$ in $\tau_{0,n+1}(v)$ ending at $\ell(v_1 \cdots v_i)$. 
Since $\tau_{0,n+1}(v_i)$ ends with $w_n$, there is an occurrence of $w_n$ in $\tau_{0,n+1}(v)$ ending at the same position $\ell(v_1 \cdots v_i)$.
Hence, if $\ell(v_1 \cdots v_i) \ne \ell(v_1 \cdots v_{i'}) \pmod{|u|}$ for two distinct $i,i'$, then $w_n$ occurs in the infinite word $u^\infty$ at positions that are different modulo $|u|$ since $\tau_{0,n+1}(v)$ occurs in $u^P$. As $|w_n| \ge 4|u|$, we may apply Fine and Wilf's theorem (e.g. see~\cite[Proposition 1.3.5]{Lothaire1983}), which yields that $u$ is not a primitive word, contradicting the definition of $u$.
Therefore, all prefixes $v_1 \cdots v_i$ of $v$ produce the same residue $\ell(v_1 \cdots v_i) \pmod{|u|}$.
This implies $|\tau_{0,n+1}(v_i)| = \ell(v_1 \cdots v_i) - \ell(v_1 \cdots v_{i-1}) = 0 \pmod{|u|}$ for all $2 \le i \le |v|$, proving~\eqref{eq: the claim on length mod u} as desired.

Next, since $\tau_{0,n+1}(v)$ occurs in $u^P$ and using the congruences in~\eqref{eq: the claim on length mod u} for all $2 \le i \le k$, every $\tau_{0,n+1}(v_i)$ begins with the same letter, say $\lambda$. 
Now, from the definition of Arnoux--Rauzy substitutions, $\tau_m(a)$ begins with $a$ for all $a \in \cA$ and $m \ge 0$, so $\tau_{0,n+1}(v_i)$ also begins with $v_i$ for all $i$.
Therefore, $v_i = \lambda$ for all $2 \leq i \le |v|$.
This implies that $v = v_1 \, \lambda^{|v|-1}$.

\textbf{Step 3}: Therefore, we obtain that $\lambda^{C_1+2} \in \cL(\infw{x}^{(n+1)})$ since $|v| \ge C_1+3$.
We now prove that this contradicts the boundedness hypothesis on weak partial coefficients. Note that the Arnoux--Rauzy word $\infw{x}^{(n+1)}$ also has weak partial coefficients bounded by $C_1$.
So we have that $\infw{x}^{(n+1)} = \sigma_a^r (\sigma_b(\infw{x}^{(n+r+2)}))$ for some Arnoux--Rauzy substitutions $\sigma_a$ and $\sigma_b$, $a \ne b$ and some $1 \le r \le C_1$, where $\infw{x}^{(n+r+2)}$ is the Arnoux--Rauzy generated by $(\tau_m)_{m > n+r+2}$.
We have 
\[  \sigma_a^r \sigma_b : a \mapsto aba^r,\, b \mapsto ba^r,\, c \mapsto ca^rba^r   \] 
for $c \not\in\{ a,b\}$.
As $\infw{x}^{(n+1)}$ is an infinite concatenation of words of the form $\sigma_a^r(\sigma_b(c))$, with $c \in \cA$, we deduce that $c^{r+2} \notin \cL(\infw{x}^{(n+1)})$ for all $c \in \cA$. 
Therefore, as $\lambda^{C_1+2} \in \cL(\infw{x}^{(n+1)})$, we must have $C_1+2 < r+2$, contradicting that $r \le C_1$. 
We conclude that the word $\infw{x}$ is $P$-power free.
\end{proof}

We are now ready to prove~\cref{thm:mainAR} in the ternary case. Notice that the necessary condition holds for every alphabet. 

\begin{proof}[Proof of \cref{thm:mainAR} in the ternary case]
We first turn to the necessary condition. Then $\infw{x}$ admits the following $\cS$-adic expansion
\[  \infw{x} = \lim_{n\to\infty}
        \sigma_{b_0}^{k_0} \circ \sigma_{b_1}^{k_1} \circ \cdots \circ \sigma_{b_n}^{k_{n}}(a),
\] for some $a\in \cA$, $b_i \in \cA$, where $(k_n)_{n\ge 0}$ is the sequence of weak partial quotients. For $c\in \cA$, define $x_n^{(c)}=\sigma_{b_0}^{k_0} \circ \sigma_{b_1}^{k_1}  \circ \cdots \circ \sigma_{b_n}^{k_{n}}(c)$. 
For each $n\ge 0$ and a letter $c\neq b_{n+1}$, we have 
\[ x_{n+1}^{(c)}=\sigma_{b_0}^{k_0} \circ \sigma_{b_1}^{k_1} \circ \cdots \circ \sigma_{b_n}^{k_{n}} \circ \sigma_{b_{n+1}}^{k_{n+1}}(c) = x_n^{(c)}(x_n^{(b_{n+1})})^{k_{n+1}}.\] 
Since $x_{n+1}^{(c)} \in \cL(\infw{x})$ and $\infw{x}$ is an aperiodic uniformly factor-balanced word, \cref{factorBal-powerFree:main_prop} yields that $(k_n)_{n\ge 0}$ is bounded.

We now prove the sufficient condition. As announced before, we make use of \cref{thm:main_thm_weaker_hyp}. All necessary hypothesis are verified since \begin{enumerate}
    \item[(i)] Items (1) and (2) of \cref{prop:Sadic_props_for_LR} hold, by \cref{lem:images_mainAR}. 
    \item[(ii)] Power-freeness holds, by \Cref{lem:powerfree_mainAR}.
    \item[(iii)] Letter-balancedness at each level holds, by \cref{thm:AR_letter-balanced}. 
\end{enumerate}
Therefore, since the constant of letter-balancedness is independent from the considered level, the word is uniformly factor-balanced. 
\end{proof}

\section{Factor-balancedness and factor complexity}
\label{sec:complexity}

The goal of this section is to exhibit certain links between the (uniform) factor-balancedness and factor complexity. 
It is well known that linearly recurrent words possess a linear factor complexity~\cite{Durand-1999}, so it seems natural to also investigate this matter regarding factor-balancedness. Let us mention that the previous work~\cite{BertheCecchi} of Berthé and Cecchi Bernales also focuses on infinite words with low complexity, since dendric words have at most linear complexity. 
In this paper, we prove that substitutive factor-balanced words have linear factor complexity (see~\cref{UFB&substitutive=>uniformly_recurrent}), even if the fixed point is not uniformly recurrent. 
Remember from a result of Pansiot~\cite{Pansiot} that substitutive words can reach quadratic factor complexity. 
Finally, in~\cref{sec:exponential sub complexity}, we also provide an example of a factor-balanced word with exponential factor complexity using Toeplitz words, whose construction is recalled and illustrated in~\cref{sec:Toeplitz}. 

\subsection{Non-primitive substitutive words}

In the following, we need the next result: the proof of the first item follows from~\cite[Lemma 2.8]{Rust_tameness} and~\cite[Lemma 3.8]{Shimomura_tameness}; the second from~\cite[Theorem A]{Shimomura_tameness}; and the last from~\cite[Theorem 1, Item (i)]{Pansiot}.

\begin{proposition}[{\cite{Rust_tameness,Pansiot,Shimomura_tameness}}]
    \label{general_substitutions:basic_stuff}
    Let $\tau$ be a substitution.
    \begin{enumerate}
        \item If $\cL(\tau)$ is $P$-power-free for some $P \in \NN$, then $\tau$ is tame.
        \item If $\tau$ is tame and quasi-primitive, then $\cL(\tau)$ is uniformly recurrent.
        \item If $\tau$ is growing and quasi-uniform, then any fixed point $\infw{x}$ of $\tau$ has linear-growth factor-complexity.
    \end{enumerate}
\end{proposition}

Before proving~\cref{UFB&substitutive=>uniformly_recurrent}, we first establish a variation of Pansiot's result (Item (3) of~\cref{general_substitutions:basic_stuff}), where the assumption is only based on growing letters. 

\begin{lemma}
    \label{substitutive&quasi_unif=>linear_complexity}
    Let $\infw{x}$ be a fixed point of the substitution $\tau \colon \cA^* \to \cA^*$.
    Assume that there exists $C > 0$ such that, for all growing letters $a,b \in \cC_\tau$,
    \begin{equation}
        \label{eq:hip:substitutive&quasi_unif=>linear_complexity}
    |\tau^n(a)| \leq C \, |\tau^n(b)|
    \enspace \text{for every $n \ge 1$.}    
    \end{equation}
    If $\tau$ is tame, then $\infw{x}$ has linear factor complexity.
\end{lemma}

\begin{proof}
The strategy is as follows.
We construct a new substitution $\sigma$ having a fixed point $\infw{y}$ such that $\infw{x} = \pi(\infw{y})$ for some substitution $\pi$, in such a way that $\sigma$ is growing.
We then show that Hypothesis \eqref{eq:hip:substitutive&quasi_unif=>linear_complexity} implies that $\sigma$ is quasi-uniform, enabling us to use Item (3) of~\cref{general_substitutions:basic_stuff}; hence $\infw{y}$ has linear-growth factor complexity.
As $\infw{x} = \pi(\infw{y})$, the linear-growth factor complexity is inherited by $\infw{x}$.

Let $\cR$ be the set of return words to growing letters $\cC_\tau$ in $\infw{x}$, i.e., $u \in \cR$ if and only if $u$ occurs at a position $i \ge 0$ of $\infw{x}$ such that $x_i, x_{i+|u|} \in \cC_\tau$ and $x_j \notin \cC_\tau$ for all $i < j < i +|u|$.
Since $\tau$ is tame (i.e., $\cB_\tau^* \cap \cL(\tau)$ is finite), the set $\cR$ is finite, and furthermore, $\infw{x}$ admits a factorization $\infw{x} = w_0 w_1 w_2 w_3 \cdots$, where $w_j \in \cR$ for all $j \ge 0$. (Note that we use that $x_0 \in \cC_\tau$.)
Let $\pi \colon \cB \to \cR$ be any bijection between a finite alphabet $\cB$ and $\cR$, and extend $\pi$ to a substitution $\pi \colon \cB^* \to \cA^*$.
Now, since $\tau$ is tame, $\tau(a)$ begins with a growing letter for any $a \in \cC_\tau$.
Thus, for every $w \in \cR$, $\tau(w)$ factorizes into a concatenation of words in $\cR$.
This induces a substitution $\sigma \colon \cB^* \to \cB^*$ such that $\pi(\sigma(b)) = \tau(\pi(b))$ for all $b \in \cB$, as shown in the commutative diagram below:
\[ \begin{tikzcd}
\cB^* \arrow{r}{\sigma} \arrow[swap]{d}{\pi} & \cB^*  \arrow{d}{\pi} \\%
\cA^*  \arrow{r}{\tau}& \cA^* 
\end{tikzcd}
\]  
An inductive argument then shows that
\begin{equation}
\label{eq:diagram:substitutive&quasi_unif=>linear_complexity}
    \pi(\sigma^n(b)) = \tau^n(\pi(b))
    \enspace \text{for all $n \ge 1$ and $b \in \cB$,}
\end{equation} 
as can be seen on the commutative diagram below:
\[\begin{tikzcd}
\cB^* \arrow[r,"\sigma"] \arrow[d,"\pi"] &
\cB^* \arrow[r,"\sigma"] \arrow[d,"\pi"] &
\cdots \arrow[r,"\sigma"] \arrow[d,"\pi"] &
\cB^* \arrow[d,"\pi"]
\\
\cA^* \arrow[r,"\tau"] & \cA^* \arrow[r,"\tau"]  & \cdots \arrow[r,"\tau"] & \cA^*
\end{tikzcd}
\]
Let $b_0 \in \cB$ be the letter corresponding to $w_0 \in R$, i.e., $\pi(b_0) = w_0$.
Then, $\sigma$ is prolongable on $b_0$ (since $\tau$ is prolongable on $x_0$), so $\infw{y} = \sigma^\omega(b_0)$ exists.
Moreover, it follows from \eqref{eq:diagram:substitutive&quasi_unif=>linear_complexity} that $\infw{x} = \pi(\infw{y})$.

Let us prove that $\sigma$ is growing and quasi-uniform.
For a substitution $\eta$, let us recall that $|\eta|$ denotes the longest image $|\eta(a)|$, with $a$ a letter. 
Let also $\ell$ be the length of the longest words in $\cR$.
Then, for $b \in \cB$ and $n \ge 1$, we can bound using \eqref{eq:diagram:substitutive&quasi_unif=>linear_complexity}
\[  |\sigma^n(b)| \leq 
    |\pi(\sigma^n(b))| =
    |\tau^n(\pi(b))| \leq 
    \ell |\tau^n|.
\]
We can also prove a lower bound as follows.
Let $c$ be the first letter of $\pi(b)$, which is growing as $\pi(b) \in \cR$.
Then, using~\eqref{eq:diagram:substitutive&quasi_unif=>linear_complexity} again, 
\[  |\sigma^n(b)| \ge 
    \frac{1}{|\pi|}|\pi(\sigma^n(b))| =
    \frac{1}{|\pi|}|\tau^n(\pi(b))| \ge 
    \frac{1}{|\pi|} |\tau^n(c)|.
\]
By Hypothesis~\eqref{eq:hip:substitutive&quasi_unif=>linear_complexity}, we have $|\tau^n(c)| \ge \frac{1}{C} |\tau^n|$, which in turn yields
\[  \frac{1}{C|\pi|} |\tau^n| \le |\sigma^n(b)| \leq 
    \ell |\tau^n|
\]
for any $b \in \cB$.
This proves that $\sigma$ is quasi-uniform.
Note that, since $\sigma$ has a growing letter (namely, $b_0$), this implies that $\sigma$ is growing.
We can then use use Item (3) of~\cref{general_substitutions:basic_stuff}, giving that $\infw{y}$, as well as $\infw{x} = \pi(\infw{y})$, has linear-growth factor complexity.
\end{proof}

The following result describes the structure of substitutive words that are factor balanced.
In particular, it shows that any such word has linear-growth complexity.
We will use the following definitions.
Let $\infw{x}=(x_n)_{n\ge 0}$ be the fixed point of a substitution $\tau \colon \cA^* \to \cA^*$ generated by a prolongable letter.
We define the set $R \subseteq \cA$ of \emph{recurrent letters} of $\infw{x}$, i.e., the letters that occur infinitely many times in $\infw{x}$.
As $\cA$ is finite, we can find $i \ge 1$ such that $x_j \in R$ for all $j \ge i$.
Thus, for any $a \in R$ occurring at a position $j \ge i$, then $\tau(x_j)$ occurs at a position $k \ge j \ge i$, which implies that $\tau(x_j)$ contains only letters in $R$.
This shows that $\tau(R) \subseteq R^*$.
Furthermore, write $u = \infw{x}_{[0,i)}$ (with $i$ taken as above) and $\tau(u) = u v$ for some non-empty word $v$.
We next define the following substitutions
\begin{equation}   
\label{eq:defi_statement:UFB&substitutive=>uniformly_recurrent}
    \sigma \colon \begin{cases}
        a & \mapsto \tau(a), \enspace \text{for $a \in R$} \\
        * & \mapsto * v
    \end{cases}
    \quad \text{and} \quad 
    \pi \colon \begin{cases}
        a & \mapsto a, \enspace \text{for $a \in R$} \\
        * & \mapsto u.
    \end{cases}
\end{equation}

\begin{proposition} \label{UFB&substitutive=>uniformly_recurrent}
Let $\infw{x}$ be the fixed point of a substitution $\tau \colon \cA^* \to \cA^*$ generated by a prolongable letter.
Consider the set $R$ defined above and the substitutions $\sigma,\pi$ defined in \eqref{eq:defi_statement:UFB&substitutive=>uniformly_recurrent}.
If $\infw{x}$ is factor-balanced, then the restriction $\tau|_R \colon R^* \to R^*$ is such that $\cL(\tau|_R)$ is uniformly recurrent and factor-balanced.
Furthermore, we have $\infw{x} = \pi(\sigma^\omega(*))$.
In particular, $\infw{x}$ has linear-growth factor complexity.
\end{proposition}

\begin{proof}
We may assume that $\infw{x}$ is not eventually periodic, as otherwise the result is immediate.
Furthermore, without loss of generality, we assume that every letter of $\cA$ occurs in $\tau^n(x_0)$ for some $n \ge 1$ (that might depend on the letter).
This implies that $\cL(\infw{x}) \supseteq \cL(\tau)$.

Let $\tau|_R\colon R^* \to R^*$ denote the restriction of $\tau$ to $R^*$.
The strategy to prove the properties about $\tau|_R$ is to use~\cref{general_substitutions:basic_stuff}.
By the same argument as that used in~\cref{recurrence&balance=>uniform_recurrence}, we can use, for each letter $a$ in $R$, that $a$ is balanced and recurrent in $\infw{x}$ to deduce that it occurs with bounded gaps in $\infw{x}$.
So, for every growing letter $a \in \cC_{\tau} \cap R$, every large enough power $\tau^n(a)$ has occurrences of every letter in $R$.
This shows that the restriction $\tau|_R$ is quasi-primitive.
Let us prove that it is tame as well.
By what we proved in the first paragraph, we have $\cL(\infw{x}) \supseteq \cL(\tau) \supseteq \cL(\tau|_R)$.
Since~\cref{factorBal-powerFree:main_prop} implies that $\infw{x}$ is $P$-power-free for some $P \in \NN$, we deduce that $\cL(\tau|_R)$ is $P$-power-free.
Thus, $\tau|_R$ is tame by Item (1) of~\cref{general_substitutions:basic_stuff}.
As $\tau|_R$ is tame and quasi-primitive, Item (2) in~\cref{general_substitutions:basic_stuff} implies that $\cL(\tau|_R)$ is uniformly recurrent.
Furthermore, since $\cL(\infw{x}) \supseteq \cL(\tau|_R)$, then $\cL(\tau|_R)$ is factor-balanced.

Next, we construct $\sigma$ and $\pi$.
Recall that $i \ge 1$ is defined so that $x_j \in R$ for all $j \ge i$.
Set $u = \infw{x}_{[0,i)}$ as in the statement.
Since $\infw{x}$ is a fixed point of $\tau$, $u$ is a prefix of $\tau(u)$, so there exists $v$ with $\tau(u) = u v$ and $v \in R^* \setminus \{\varepsilon\}$.
Let $\sigma$ and $\pi$ be as in \eqref{eq:defi_statement:UFB&substitutive=>uniformly_recurrent}, and observe that $\pi(\sigma(a)) = \tau(\pi(a))$ for every $a \in R \cup \{*\}$.
By an inductive argument, we get
\begin{equation}
\label{eq:diagram:UFB&substitutive=>uniformly_recurrent}
    \pi(\sigma^n(a)) = \tau^n(\pi(a))
    \enspace \text{for all $n \ge 1$ and $a \in R \cup \{*\}$.}
\end{equation}
Clearly, $\sigma$ is prolongable on $*$, so $\infw{y} = \sigma^\omega(*)$ is a well-defined infinite word.
Since $\pi(*) = u =\infw{x}_{[0,i)}$, $\pi(*)$ starts with $x_0$, which together with \eqref{eq:diagram:UFB&substitutive=>uniformly_recurrent} allows us to write 
\[  \infw{x} = \tau^\omega(\pi(*)) = 
    \pi(\sigma^\omega(*)) = \pi(\infw{y}).
\]

It remains to show that $\infw{x}$ has linear-growth complexity.
By the last identity, it suffices to prove that $\infw{y}$ has linear-growth complexity.
To this end, we check the hypotheses of~\cref{substitutive&quasi_unif=>linear_complexity}.
Let us first prove that $\sigma$ is tame.
We have, by~\cref{factorBal-powerFree:main_prop}, that $\infw{x}$ is $P$-power-free for some $P > 0$.
Since $\infw{x} = \pi(\infw{y})$, the same holds for $\infw{y}$.
Since $\cL(\infw{x}) \supseteq \cL(\tau)$, then $\cL(\infw{y}) \supseteq \cL(\sigma)$ by \eqref{eq:diagram:UFB&substitutive=>uniformly_recurrent}.
We deduce that $\cL(\sigma)$ is $P$-power-free, and thus that $\sigma$ is tame by Item (1) of~\cref{general_substitutions:basic_stuff}.
Next, we prove that $\sigma$ is quasi-uniform.
By classical results in Perron-Frobenius theory, it is enough to find a growing letter $\bar{a} \in \cC_\sigma$ such that, for any $b \in \cC_\sigma$, there is a power $\sigma^n(b)$ with an occurrence of $\bar{a}$; see for instance~\cite[Lemma 11]{LeroyRigoCharlier}.
Fix $\bar{a} \in \cC_\sigma \cap R$ and consider an arbitrary $b \in \cC_\sigma$.
As $\tau|_R$ is quasi-primitive and $\sigma|_R = \tau|_R$, then $\sigma|_R$ is quasi-primitive as well.
Hence, for every $b \in R \cap \cC_\sigma$ there is a power $\sigma^n(b)$ with an occurrence of $\bar{a}$.
Furthermore, $\infw{y}=\sigma^\omega(*)$, so there is a power $\sigma^n(*)$ with an occurrence of $\bar{a}$.
This proves the claim and completes the proof after using~\cref{substitutive&quasi_unif=>linear_complexity}.
\end{proof}



The next example gives a factor-balanced substitutive word that is not recurrent. Hence, in Proposition~\ref{UFB&substitutive=>uniformly_recurrent}, the auxiliary symbol $*$ is necessary to encode a possibly transient part.

\begin{example}
\label{exa:nonrecurrent_FB}
Let $\tau \colon \{0,1\}^* \to \{0,1\}^*$ be the Fibonacci substitution, i.e., $\tau(0) = 01$ and $\tau(1) = 0$.
Define $\sigma \colon \{0,1,*\}^* \to \{0,1,*\}^*$ by $\sigma(a) = \tau(a)$ for $a \in \{0,1\}$ and $\sigma(*) = * 11$.
Note that $11$ does {\em not} belong to $\cL(\tau)$. In~\cref{appendix:example}, we prove that the fixed point 
\[
\infw{x}= \sigma^\omega(*)=* 11 \tau(11) \tau^2(11)\cdots =*11000101010010010010100\cdots
\]
is factor-balanced and that none of its shifts is recurrent.
To do this, we first prove that $\infw{x}$ is letter-balanced.
Then, we show that $\infw{x}$ is \emph{Fibonacci-automatic}, i.e., it can be generated by a DFAO where the input is read in the Zeckendorf numeration system.
Finally, equipped with this property, we make use of the automatic theorem-prover \texttt{Walnut}~\cite{Walnut-2023} to prove that $\infw{x}$ is $5$-power-free and no shift of $\infw{x}$ is recurrent.
Therefore, by~\cref{thm:main_thm_weaker_hyp}, $\infw{x}$ is uniformly factor-balanced. In addition, by \cref{UFB&substitutive=>uniformly_recurrent}, it has linear complexity. 

As a final note to this example, observe that the first shift $\infw{x'}=11000\cdots$ of $\infw{x}$ has factor complexity $2n$, i.e., it belongs to the family of \emph{Rote words}. Also, this word seems to be $2$-letter-balanced, which is the smallest possible constant for such a word. Together with the word $h(\psi^{\omega}(0))$, where $\psi: 0\mapsto 0121, 1\mapsto 0, 2\mapsto 01$ and $h:0 \mapsto 01, 1 \mapsto 0, 2 \mapsto 011$, they are the only known examples of Rote words with this property\footnote{This follows from a personal communication with L. Dvořáková, E. Pelantová, and J. Shallit.}.
\end{example}

\subsection{Toeplitz words}
\label{sec:Toeplitz}

A {\em Toeplitz word} is a uniformly recurrent infinite word $\infw{x}=(x_n)_{n\ge 0} \in \cA^\NN$ such that for every $n \in \NN$ there exists a period $p \ge 1$ for which $x_{n+kp} = x_n$ for every $k \in \NN$.
We present a method to construct Toeplitz words following~\cite{DurandPerrinBook}, which is based on~\cite{GJ2000}. 
We also warn the reader that this definition is more general than that of ``Toeplitz words'' presented, for instance, in~\cite{CassaigneKarhumaki1997}. Finally, we  refer to the survey of Downarowicz~\cite{Downarowicz2005} for more details. 

Let us fix an alphabet $\cA$.
The construction depends on two parameters.
First, we need an increasing sequence $(p_n)_{n \ge 0}$ of positive integers such that 
\begin{enumerate}
    \symitem{$(\mathbf{T\!p_{1}})$}{defi:Toeplitz:divides}
    The integer $p_n$ divides $p_{n+1}$ for all $n \ge 0$.
\end{enumerate}
For $n\ge 1$, define the integer $q_n = p_n/p_{n-1}$.
The second parameter is a sequence $(W_n)_{n\ge 0}$ of sets of words with $W_n \subseteq \cA^{p_n}$ and satisfying the following: for every $n \ge 0$, 
\begin{enumerate}
    \symitem{$(\mathbf{T\!p_2})$}{defi:Toeplitz:hierarchical}
    Each $w \in W_{n+1}$ is a concatenation $w = \phi_n(w,1) \phi_n(w,2) \cdots \phi_n(w,q_n)$ of $q_n$ words from $W_n$; and
    \symitem{$(\mathbf{T\!p_3})$}{defi:Toeplitz:proper}
   There exists $w_n \in W_n$ that is a prefix of each word in $W_{n+1}$.
\end{enumerate}
Observe that the last property ensures that $w_n$ is always a prefix of $w_{n+1}$.
Since $|w_n| = p_n$ is increasing, we obtain an infinite word
\[  \infw{x} = \lim_{n\to\infty} w_n \in \cA^\NN.  \]
If we furthermore assume that the parameters satisfy, for every $n \ge 1$,
\begin{enumerate}
    \symitem{$(\mathbf{T\!p_4})$}{defi:Toeplitz:primitive}
     The set $\{ \phi_n(w,i) : 1 \leq i \leq q_n \}$ equals $W_n$ for every $w \in W_{n+1}$,
\end{enumerate}
then $\infw{x}$ is uniformly recurrent, and $\infw{x}$ is a Toeplitz word.

\begin{example}
\label{ex:PD_is_Toeplitz}
We exemplify the definitions by showing that the fixed point of the period doubling substitution is a Toeplitz word.
Let $\cA = \{0,1\}$ and $p_n = 4^n$ for $n \ge 0$, which satisfy \ref{defi:Toeplitz:divides}.
We inductively define the sets of words $W_n = \{u_0(n), u_1(n)\} \subseteq \cA^{p_n}$ as follows.
First, $u_0(0) = 0$ and $u_1(0) = 1$, and then we set
\begin{equation}
    \label{ex:PD_is_toeplitz:recurrence}
    \begin{cases}
    u_{n+1}(0)&=u_n(0)\, u_n(1)\, u_n(0)\, u_n(0), \\
    u_{n+1}(1)&=u_n(0)\, u_n(1)\, u_n(0)\, u_n(1),
    \end{cases}
\end{equation}
for $n \ge 0$.
Conditions \ref{defi:Toeplitz:hierarchical} and \ref{defi:Toeplitz:primitive} are satisfied by the very construction, and \ref{defi:Toeplitz:proper} holds as well with $w_n = u_n(0)$.
Therefore, the sequence $(u_n(0))_{n\ge 0}$ of words converges to a uniformly recurrent word $\infw{x} \in \cA^\NN$, which is a Toeplitz word.

One can check that $\infw{x}$ is exactly the fixed point of the {\em period doubling substitution} $\tau \colon \cA^* \to \cA^*$, defined by $\tau(0) = 01$ and $\tau(1) = 00$.
The equivalence of the two construction methods is a consequence of the recurrence relations in \eqref{ex:PD_is_toeplitz:recurrence} mimicking $\tau^2$, which is given by $\tau^2(0) = 0100$ and $\tau^2(1) = 0101$.
Note that in \eqref{ex:PD_is_toeplitz:recurrence} one needs to use $\tau^2$ instead of $\tau$ to guarantee that the construction satisfies \ref{defi:Toeplitz:primitive}.
\end{example}

Observe that the words $\phi_n(w,i)$ in \ref{defi:Toeplitz:hierarchical} essentially define substitutions.
We now formalize this idea, as it will be highly convenient to describe the Toeplitz construction using the language of substitutions.

Let us fix a sequence $(\cA_n )_{n \ge 0}$ of alphabets such that $\cA_0 = \cA$ and $\cA_n$ is in bijection with $W_n$ for each $n\ge 0$.
Then, these bijections translate the decompositions $w = \phi_n(w,1) \phi_n(w,2) \cdots \phi_n(w,q_n)$ in \ref{defi:Toeplitz:hierarchical} into a map $\tau_n \colon \cA_{n+1} \to \cA_n^{q_n}$.
The map $\tau_n$ extends uniquely into a substitution $\tau_n \colon \cA_{n+1}^* \to \cA_n^*$ such that 
\begin{enumerate}
    \symitem{$(\mathbf{T\!p'_1})$}{defi:Toeplitz:sadic:divides} 
    We have $|\tau_n(a)| = \frac{p_n}{p_{n-1}}$ for every $n \ge 1$ and every $a \in \cA_{n+1}$.
\end{enumerate}
Conditions \ref{defi:Toeplitz:proper} and \ref{defi:Toeplitz:primitive} are then equivalent to, for every $n \ge 0$,
\begin{enumerate}
    \symitem{$(\mathbf{T\!p'_3})$}{defi:Toeplitz:sadic:proper} 
    There exists $a_n \in \cA_n$ such that $\tau_n(a)$ begins with $a_n$ for every $a \in \cA_{n+1}$.
    \symitem{$(\mathbf{T\!p'_4})$}{defi:Toeplitz:sadic:primitive} 
    Every $b \in \cA_n$ occurs in $\tau_n(a)$ for every $a \in \cA_{n+1}$.
\end{enumerate}
Condition \ref{defi:Toeplitz:sadic:proper} is usually called {\em left-properness}, and \ref{defi:Toeplitz:sadic:primitive} {\em strong primitiveness}.

Therefore, a Toeplitz word $\infw{x}$ is completely described by a sequence $(\tau_n \colon \cA_{n+1} \to \cA_n^{q_n})_{n \ge 0}$ of substitutions satisfying \ref{defi:Toeplitz:sadic:divides}, \ref{defi:Toeplitz:sadic:proper} and \ref{defi:Toeplitz:sadic:primitive}, by the formula
\[  \infw{x} = \lim_{n \to \infty} \tau_0 \circ \tau_1 \circ \dots \circ \tau_n(a_{n+1}). \]
Observe that, for each $m\in\NN$, the shifted sequence $(\tau_n)_{n \ge m}$ also satisfies the Toeplitz conditions, so it defines a Toeplitz word
\begin{equation}
    \label{defi:Toeplitz:levels}
    \infw{x^{(m)}} = \lim_{n \to \infty} \tau_m \circ \tau_{m+1} \circ \dots \circ \tau_n(a_{n+1})
    \enspace \text{for each $m \ge 0$.}
\end{equation}
The intermediate words $(\infw{x^{(m)}})_{m\ge 0}$ play an important role in understanding $\infw{x}$.
They satisfy the relations
\begin{equation}
    \label{defi:Toeplitz:xm_to_xn}
    \infw{x^{(m)}} = \tau_m \circ \tau_{m+1} \circ \dots \circ \tau_{n-1}(\infw{x^{(n)}}),
    \enspace \text{for all $n > m \ge 0$.}
\end{equation}
This shows in particular that $\infw{x^{(m)}}$ is an infinite concatenation of words $\tau_m \circ \tau_{m+1} \circ \dots \circ \tau_{n-1}(a)$ for $a \in \cA_n$.

\begin{remark}
    We carefully warn the reader that the notation $\infw{x^{(m)}}$ in the framework of Toeplitz does not refer to the notation of inducing tuples from the previous sections of the paper. Since we do not use both of these objects simultaneously, what the notation $\infw{x^{(m)}}$ means should be clear regarding the context. 
\end{remark}

\subsection{Exponential factor complexity}
\label{sec:exponential sub complexity}
Our aim is to prove~\cref{thm: FB word with exp complexity}.
We thus construct a factor-balanced Toeplitz word $\infw{x}$ whose complexity function $p_\infw{x}$ satisfies 
\[  \lim_{n\to\infty} \frac{1}{n} \log p_\infw{x}(n) > 0, \] i.e., $\infw{x}$ has a positive entropy. 

Let $\cA_0$ be an alphabet. Consider two sequences of positive integers $(k_n)_{n \ge 0}$ and $(r_n)_{n \ge 0}$ with $r_n\leq 2k_n$ and that will be fixed later. We inductively construct $\tau_n \colon \cA_{n+1}^* \to \cA_n^*$ as follows.

First, we simply set $\cA_1 = \cA_0$ and $\tau_0 \colon \cA_1 \to \cA_0$ to be the identity map, i.e., $\tau_0(a) = a$ for all $a \in \cA_1$, and we extend it to a substitution by concatenation.
Suppose that $\tau_{n-1}$ has been defined.
Fix a primitive word $v_n u_n \in \cA_n^*$ with $|u_n| = |v_n| = r_n$.

Let $W_n$ be the set of all words $w \in \cA_n^*$ such that:
\begin{enumerate}
    \item The word $w$ begins with $u_n$ and ends with $v_n$;
    \item Each $a \in \cA_n$ occurs exactly $k_n$ times in $w$, i.e., $|w|_a=k_n$;
\end{enumerate}
Note that the second condition implies that $|w| = k_n |\cA_n|$ for all $w \in W_n$.
We let $\tau_n \colon \cA_{n+1} \to W_n$ to be any bijection between the alphabet $\cA_{n+1}$ and $W_n$, and then extend $\tau_n$ to a substitution $\tau_n \colon \cA_{n+1}^* \to \cA_n^*$.

\begin{example}
    Let us describe how the first step works on the simple example where $r_1=2$, $k_1=4$, $\cA_{1}=\{0,1\}$, $u_1=01$, and $v_1=10$. Then $w \in W_1$ if $w=01w'10$, $|w'|_0=|w'|_1=2$. Thus, there are six choices for $w'$: $0011,0101,0110,1001,1010,1100$. Finally, we define $\cA_2=\{0,1,\ldots, 5\}$ and $\tau_1:0\mapsto 0011,1\mapsto 0101,\ldots, 5\mapsto 1100$.  
\end{example}

It is not difficult to see that the sequence $(\tau_n )_{n \ge 0}$ of substitutions defined above satisfies the Toeplitz conditions, so they define Toeplitz words $\infw{x^{(m)}}$ according to \eqref{defi:Toeplitz:levels}.
We set $\infw{x} = \infw{x^{(0)}}$. 


\begin{theorem}
    \label{Toep_construct:factor_bal&positive_entropy}
    The word $\infw{x}$ is factor-balanced. Moreover, there exists a choice of sequences $(r_n,k_n)$ such that the word $\infw{x}$ has exponential-growth factor complexity.
\end{theorem}

Observe that~\cref{thm: FB word with exp complexity} directly follows from the previous statement.
We prove~\cref{Toep_construct:factor_bal&positive_entropy} using a series of intermediate results.~\cref{Toep_construct:factor_balanced} proves the first statement of the theorem, and~\cref{Toeplitz:exponential_complexity} the second. 

\begin{lemma}
    \label{Toep_construct:letter_balanced}
For every $m \ge 0$, $\infw{x^{(m)}}$ is $(4k_m)$-letter-balanced.
\end{lemma}
\begin{proof}
Fix $m \ge 0$, let $u \in \cL(\infw{x^{(m)}})$ and $a \in \cA_m$. It is enough to prove that 
\begin{equation}
    \label{proof:main_claim:Toep_construct:letter_balanced}
    \left| |u|_a - k_m\, \frac{|u|}{|\cA_m|} \right| 
    \leq 2k_m,
\end{equation}
as then for any other $v \in \cL(\infw{x^{(m)}})$ of length $|v| = |u|$, the previous inequality yields
\[  \big| |u|_a - |v|_a \big|  \leq 
    \left| |u|_a - k_m\, \frac{|u|}{|\cA_m|} \right| + 
    \left| k_m\, \frac{|v|}{|\cA_m|} - |v|_a \right| \leq 4k_m. \]

Let us prove \eqref{proof:main_claim:Toep_construct:letter_balanced}.
Observe that, by construction, $\infw{x^{(m)}}$ is an infinite concatenation of words of the form $\tau_m(b)$ for $b \in \cA_{m+1}$.
In particular, since $u$ occurs in $\infw{x^{(m)}}$, it must occur in $\tau_m(b_1 \cdots b_\ell)$ for some word $b_1 \cdots b_\ell \in \cA_{m+1}^*$ with $\ell \ge 1$.
We take such a word of minimal length $\ell$. By construction, for any $b \in \cA_{m+1}$, $\tau_m(b)$ satisfies \[\begin{cases}
    |\tau_m(b)|=k_m |\cA_m|, \\
    |\tau_m(b)|_a=k_m. 
\end{cases}\]
Therefore, since $u$ occurs in $\tau_m(b_1 \cdots b_\ell)$, we have the following upper bounds \[ \begin{cases}
   |u| \leq \ell k_m |\cA_m|, \\
    |u|_a \leq \ell k_m.
\end{cases}\] 
Now, by the minimality of $\ell$, either $\ell \leq 2$ or $\tau_m(b_2 b_3 \cdots b_{\ell-1})$ occurs inside $u$.
In both cases, we obtain the lower bounds \[ \begin{cases}
   |u| \geq (\ell-2) k_m |\cA_m|,  \\
    |u|_a \geq (\ell-2) k_m.
\end{cases}\] 
From the previous inequalities we deduce that \[ \begin{cases}
   (\ell-2) k_m \leq |u|_a \leq \ell k_m,  \\
    \frac{|u|}{k_m |\cA_m|} \leq \ell \leq 
    \frac{|u|}{k_m |\cA_m|} + 2.
\end{cases}\] Therefore, plugging the second inequalities in the first, we obtain 
\[
\frac{|u|}{|\cA_m|} - 2 k_m \leq |u|_a \leq 
\frac{|u|}{|\cA_m|} + 2k_m.
\]
This proves \eqref{proof:main_claim:Toep_construct:letter_balanced} and thus completes the proof of the lemma.
\end{proof}

To prove the first part of~\cref{Toep_construct:factor_bal&positive_entropy}, we now use the following result of Poirier and Steiner~\cite{PoirierSteiner}.
In their paper, it is stated in terms of symbolic dynamics; we state an equivalent version for infinite words. 

\begin{proposition}[{\cite[Theorem 4.1]{PoirierSteiner}}]
    \label{Poirier&Steiner:factor_balance}
    Let $(\tau_n \colon \cA_{n+1}^* \to \cA_n^* )_{n \ge 0}$ be a sequence of substitutions that are left-proper and primitive.
    Let $\infw{x^{(m)}} \in \cA_m^\NN$ be the limit word of the shifted sequence $(\tau_n)_{n \ge m}$, for every $m \ge 0$.
    If each $\infw{x^{(m)}}$ is letter-balanced, then $\infw{x^{(0)}}$ is factor-balanced.
\end{proposition}

\begin{lemma}
    \label{Toep_construct:factor_balanced}
    The word $\infw{x}$ is factor-balanced.
\end{lemma}

\begin{proof}
The substitutions $\tau_n$ are left-proper by \ref{defi:Toeplitz:sadic:proper} and primitive by \ref{defi:Toeplitz:sadic:primitive}.
These two conditions together with~\cref{Toep_construct:letter_balanced} are precisely the hypothesis of~\cref{Poirier&Steiner:factor_balance}, whose conclusion is that $\infw{x}$ is factor-balanced.
\end{proof}

We now turn to the task of estimating the factor complexity of $\infw{x}$.
The first step is to prove a recognizability result.
For this, we need the words $v_n u_n$ from the construction. 

\begin{lemma}
\label{Toeplitz:morphisms_are_injective}
For every $m > n \ge 0$, the map $\tau_{n,m}:\cA_m \rightarrow \cA_n$ is injective.

\end{lemma}
\begin{proof}
Fix some $n\ge 0$ and let us prove the result by induction on $m$. 
When $m = n + 1$, then $\tau_{n,m} = \tau_n$, so the base case follows directly from the construction of $\tau_n$, which was done by requiring $\tau_n$ to map bijectively $\cA_n$ onto $W_n$. Assume $m  n + 1$ and take $a,a' \in \cA_m$ with $\tau_{n,m}(a) = \tau_{n,m}(a')$, and let us prove that $a=a'$. Therefore, we have $\tau_{n,m-1}(\tau_{m-1}(a)) = \tau_{n,m-1}(\tau_{m-1}(a'))$. By induction hypothesis, $\tau_{n,m-1}$ is injective, then $\tau_{m-1}(a) = \tau_{m-1}(a')$, and the result follows since $\tau_{m-1}$ is chosen to be a bijection.
\end{proof}

We need the next estimating result.

\begin{lemma}
\label{estimates_multinomial}
For $k,d \in \NN$, the multinomial coefficient $\binom{dk}{k,\dots,k} \coloneqq 
    \frac{(dk)!}{(k!)^d}$ satisfies
\[  \binom{dk}{k,\dots,k} =
    \bigl(1 + o(1)\bigr) \frac{\sqrt{d}}{(2k\pi)^{(d-1)/2}} d^{dk}.
    \]
\end{lemma}

\begin{proof}
This is the classical estimate of $\binom{dk}{k,\dots,k}$ when using Stirling's approximation to obtain 
\begin{equation}
\label{eqn:approx by Stirling}
    \log\left(\binom{dk}{k,\dots,k}\right)=dk\log(d)+\frac{1}{2}\log(d)-\frac{d-1}{2}\log(2\pi k)+O\left(\frac{1}{k}\right),
\end{equation}
and the result follows.
\end{proof}

The second part of~\cref{Toep_construct:factor_bal&positive_entropy} now follows.
 
\begin{lemma}\label{Toeplitz:exponential_complexity}
   There exists a choice of sequences  $(r_n,k_n)_{n \ge 0}$ such that the word $\infw{x}$ has exponential-growth factor complexity.
\end{lemma}

\begin{proof}
Let $\ell_n=|\cA_n|$, $u_n=01\cdots (\ell_n-1)$, and $v_n=(\ell_n-1)(\ell_n-2)\cdots 0$. Therefore, $r_n=\ell_n$ for all $n\geq 0$. By construction,  we have $p_{n+1}=p_nk_n\ell_n$. Then, by~\cref{Toeplitz:morphisms_are_injective}, $p_{\infw{x}}(p_{n+1}) \ge \ell_{n+1}$. Therefore, our goal is to prove that $\frac{1}{p_n} \log(\ell_n) >c$ for all $n\geq 0$ and for some constant $c>0$. 

Remember that $\ell_{n+1}$ depends only on the choice of $k_n$, and that the choice of $k_{n+1}$ is independent of $k_n$. 
Suppose that all $k_{i}$'s, $i\leq n-1$, have been chosen. If $k_n$ is sufficiently large, $\ell_{n+1}$ corresponds to the number of words $w\in \cA_n^*$ such that \begin{enumerate}[label=(\roman*)]  
    \item The word $w$ starts with $u_n$ and ends with $v_n$.  
    \item We have $|w|_a = k_n$ for all $a \in \cA_n$. In particular, we have $|w| = k_n|\cA_n|$.  
\end{enumerate}  
Since $|u_n|_a=|v_n|_a=1$ for all $a\in \cA_n$, we have that $w=u_n w' v_n$ for some $w'$ such that $|w'|_a=(k_n-2)$ for all $a \in \cA_n$. Therefore, the number of such $w'$ is exactly the multinomial coefficient $\binom{\ell_n (k_n-2)}{k_n-2,\ldots,k_n-2}$. Thus, we have \begin{align*}
    \ell_{n+1}=(1+o(1))\left(\ell_n^{\ell_n (k_n-2)}\frac{\sqrt{\ell_n}}{(2(k_n -2)\pi)^{(\ell_n-1)/2}}\right).
\end{align*}
Then for sufficiently large $k_n$, we have 
\begin{equation}
 \label{eqn: bound on the l's}
 \ell_{n+1} > \ell_n^{\ell_n k_n} \frac{1}{(2\pi\ell_n^2k_n)^{\ell_n}}.
\end{equation}
Using the bound in~\eqref{eqn: bound on the l's}, we have \begin{align*}
    \frac{1}{p_{n+1}}\log(\ell_{n+1}) &\geq \frac{1}{p_nk_n\ell_n}\log \left(\ell_n^{\ell_n k_n} \frac{1}{(2\pi\ell_n^2k_n)^{\ell_n}}\right), \\
    & \geq \frac{\log(\ell_n)}{p_n}-\frac{1}{p_nk_n}\log(2\pi \ell_n^2k_n), \\
    & \geq \frac{\log(\ell_n)}{p_n}-\frac{2}{p_nk_n}\log(\ell_n)-\frac{1}{p_nk_n}\log(2\pi k_n), \\
    & \geq \frac{\log(\ell_n)}{p_n}\left(1-\frac{2}{k_n}\right)-\frac{3}{p_n}. 
\end{align*}
Thus, by iterating the previous inequality, we obtain for all $n\geq 2$ \begin{align*}
    \frac{1}{p_{n}}\log(\ell_{n}) \geq \frac{1}{p_2}\log(\ell_2) \prod_{2\leq i\leq n-1} \left(1-\frac{2}{k_i}\right) - S_n,
\end{align*}
where \[S_n:=\sum_{2\leq i \leq n-1} \frac{3}{p_{i}} \prod_{i+1 \le j \le n-1}\left(1-\frac{2}{k_{j}}\right).\] 
Since $p_i=p_{i-1}\ell_{i-1}k_{i-1}\geq 2p_{i-1}$, we have $p_i\geq 2^i$ for all $i\geq 2$. And we have $S_n\leq \sum_{2\leq i \leq n-1}\frac{3}{p_i}$, then the sum $S=\lim_{n\to +\infty}S_n$ is finite. 
Now, we choose $(k_n)_{n\ge 2}$ such that $\prod_{2\leq i\leq n} \left(1-\frac{2}{k_i}\right)$ is finite and non-zero, e.g., $k_n \ge n^2$ for all $n\ge 2$. Then we have \begin{align*}
    \frac{1}{p_{n}}\log(\ell_{n}) \geq \frac{1}{p_2}\log(\ell_2) \prod_{i\geq 2} \left(1-\frac{2}{k_i}\right) - S,
\end{align*} since $\prod_{2\leq i \leq n-1} \left(1-\frac{2}{k_i}\right)\geq \prod_{i\geq 2} \left(1-\frac{2}{k_i}\right)$ and $S_n\leq S$. 
Finally, we choose $|\cA|=\ell_1=\ell_0$ and $k_1$ as below.
Recall that $p_1=1$, $p_2 = p_1\ell_1k_1 = |\cA| k_1$, and by~\eqref{eqn:approx by Stirling}, we have
\begin{align*}
\frac{1}{p_2} \log \ell_2 = \frac{1}{|\cA|k_1} \bigg( |\cA|(k_1-2)\log(|\cA|)&+\frac{1}{2}\log(|\cA|)  \\ & - \left. \frac{|\cA|-1}{2}\log(2\pi (k_1-2))+O\left(\frac{1}{k_1}\right) \right)
\end{align*}
So, first choose $\cA$ such that $\log|\cA| \prod_{i\geq 2} \left(1-\frac{2}{k_i}\right) > S$.
Then, as $k_1 \to\infty$, we obtain 
\[
\frac{1}{p_2} \log \ell_2 = 
\log|\cA|\prod_{i\geq 2} \left(1-\frac{2}{k_i}\right)+O\left(\frac{1}{k_1}\right).
\]
In the end, we may choose $k_1$ such that
\[
    \frac{1}{p_n} \log p_{\infw{x}}(p_n) \ge \frac{1}{p_{n}} \log |\cA_n| \ge \frac{1}{p_2}\log(\ell_2) \prod_{i\geq 2} \left(1-\frac{2}{k_i}\right) - S>0. 
\]
Therefore, the sequence $\bigl(\frac{1}{p_{n}} \log p_{\infw{x}}(p_{n})\bigr)_{n\ge 1}$ remains bounded away from zero, which means that $\infw{x}$ has exponential-growth factor complexity, as desired.
\end{proof}


\bibliographystyle{alpha}
\bibliography{biblio.bib}

\appendix

\phantomsection

\addcontentsline{toc}{section}{Appendices}

\addtocontents{toc}{\protect\setcounter{tocdepth}{0}}

\section{Tribonacci linearly recurrent constant in \texorpdfstring{\cref{ex: constant for the Tribonacci word}}{}}
\label{appendix: Trib}

In this section, we use the free software \texttt{Walnut} to establish the linear recurrence constant given in~\cref{ex: constant for the Tribonacci word}. 
Mousavi~\cite{Mousavi2016} designed \texttt{Walnut} to automatically decide the truth of assertions about many properties for a large family of words.
We assume that the reader is familiar with the formalism of the software; if not, we refer them to~\cite{Walnut-2023} for a complete introduction and a survey of the combinatorial properties that can be checked. 

From~\cref{ex: constant for the Tribonacci word}, recall the Tribonacci substitution $\tau \colon \cA^* \to \cA^*$ by $a \mapsto ab$, $b \mapsto ac$, and $c \mapsto a$ and its unique fixed point $\infw{t}=(t_i)_{i\geq 0}$ called the \emph{Tribonacci word}.
The \emph{Tribonacci numeration system} is a positional numeration system for integers based on the sequence of \emph{Tribonacci numbers}, i.e., the sequence $(T_n)_{n\ge 0}$ of integers defined by $T_0=1$, $T_1=2$, $T_2=4$, and $T_{n+3}=T_{n+2}+T_{n+1}+T_n$ for all $n\ge 0$.
The Tribonacci word $\infw{t}$ is then \emph{Tribonacci-automatic}, i.e., it can be generated by a DFAO where the input is read in the Tribonacci numeration system.
For more on generalized automatic sequences, for instance see~\cite{Rigo-2000,Maes-Rigo}. In \texttt{Walnut}, the Tribonacci numeration system and the Tribonacci word $\infw{t}$ are respectively stored under the name \verb|trib| and \verb|TR|. 

Let us recall that a word $\infw{x}$ is linearly recurrent with constant $L\ge 0$ if, for all $u \in \cL(\infw{x})\setminus \{\varepsilon\}$ and all $w \in \mathcal{R}_{\infw{x}}(u)$, we have $|w|\leq L|u|$. 
In particular, $\infw{x}$ is linearly recurrent with constant $L$ if every length-$(L+1)n$ factor of $\infw{x}$ contains every element of $\cL_n(\infw{x})$. We rather use this second description to compute the linear recurrence constant via \texttt{Walnut}. We define $R_{\infw{x}}(w)$ to be the smallest integer $k$ such that every factor of length $k$ of $\infw{x}$ contains $w$, or $\infty$ is no such $k$ exists. We also define the recurrence function as $R_{\infw{x}}(n) = \max_{w \in  \cL_n(\infw{x})} R_{\infw{x}}(w)$. Therefore, if $\infw{x}$ is linearly recurrent with constant $L\ge 0$, we have \begin{align} \label{eq:linrec L+1}
    L+1=\lim_{n\to \infty}\frac{R_{\infw{x}}(n)}{|n|}.
\end{align}

To establish the linear recurrence constant in~\cref{ex: constant for the Tribonacci word}, we will need a formula to test equality of factors in $\infw{t}$, i.e., for all integers $i,j,n$, $\texttt{tribeqfac(i,j,n)}$ defined by
\begin{verbatim}
def tribeqfac "?msd_trib Au,v (u>=i & u<i+n & u+j=v+i) => TR[u]=TR[v]": 
\end{verbatim}
returns \texttt{TRUE} if and only if $t_i \cdots t_{i+n-1}=t_j\cdots t_{j+n-1}$.
It produces an automaton of $26$ states.
The function $\texttt{hasall(j,m,n)}$ defined by
\begin{verbatim}
def hasall "?msd_trib Ai Ek $tribeqfac(i,k,n) & k>=j & k+n<=j+m":
\end{verbatim}
returns \texttt{TRUE} if and only if $t_j\cdots t_{j+m-1}$ contains all length-$n$ factors.
It produces an automaton of $710$ states.
The command $\texttt{recur(m,n)}$ defined by
\begin{verbatim}
def recur "?msd_trib Aj $hasall(j,m,n)":
\end{verbatim}
returns \texttt{TRUE} if and only if every length-$m$ factor of $\infw{t}$ contains all factors of length $n$.  
It produces an automaton of $34$ states. The recurrence function $R_{\infw{t}}(n)$ is then given by
\begin{verbatim}
def rf "?msd_trib $recur(z,n) & ~$recur(z-1,n)":
\end{verbatim}

\begin{lemma}
 The linear recurrence constant of the Tribonacci word $\infw{t}$ is in $[9,10]$. 
\end{lemma}
\begin{proof}
Using the commands written above, we can run the following in \texttt{Walnut}:
\begin{verbatim}
def check1 "?msd_trib An,z (n>=1 & $rf(n,z)) => z<=11*n":
def check2 "?msd_trib An,z (n>=1 & $rf(n,z)) => z<=10*n":
\end{verbatim}
The first returns \texttt{TRUE} and the second \texttt{FALSE}, which proves the lemma.
\end{proof}

To get the precise value of the linear recurrent constant, more work is needed. 
Let $(U_n)_{n\ge 0}=1,2,4,8,15,28,52,\ldots$ be the sequence of integers satisfying the recurrence relation $U_n=2U_{n-1}-U_{n-4}$ for all $n\ge 4$ with initial conditions $(U_n)_{0\le n\le 3}=1,2,4,8$.
Note that these integers are represented in the Tribonacci numeration system by the sequence of binary words starting with $1,10,100,1001,10010,100100,1001001,\ldots$.
The following commands
\begin{verbatim}
reg link msd_trib msd_trib "[0,0]*[1,1]([0,0][0,0][0,1])
                                    *(()|[0,0]|[0,0][0,0])":
reg adjtrib msd_trib msd_trib "[0,0]*[0,1][1,0][0,0]*":
def tn4 "?msd_trib Ej,k,l $adjtrib(n,j) & $adjtrib(j,k) 
                                & $adjtrib(k,l) & $adjtrib(l,z)":
def adju "?msd_trib Eu,t $link(u,x) & $adjtrib(x,t) & $link(t,y)":
\end{verbatim}
respectively provide an automaton that accepts the Tribonacci representations of the pairs $(T_n,U_n)$, $(T_n,T_{n+1})$, $(T_n,T_{n+4})$, and $(U_n,U_{n+1})$ in parallel (possibly with leading $0$'s).

\begin{lemma}
Let $R_{\infw{t}}$ be the recurrence function in the Tribonacci word $\infw{t}$.
If $U_i < n \leq U_{i+1}$ for integers $i,n$, then $R_{\infw{t}}(n)=T_{i+4}+n-1$. 
\end{lemma}
\begin{proof}
The following command 
\begin{verbatim}
eval recurformula "?msd_trib Ax,y,z,v,n (n>=1 & $adju(x,y) & n>x & n<=y 
                & $rf(n,z) & $tn4(x,v)) => z+1=v+n":
\end{verbatim}
checks the formula for the recurrence function $R_{\infw{t}}(n)$ written in the statement and returns \texttt{TRUE}.
\end{proof}

Recall that for functions $f,g \colon \NN \to \RR$, with $f(n), g(n) > 0$ for all $n \ge 0$, we write $f \sim g$ if $\lim_{n\to\infty} f(n)/g(n) = 1$.

\begin{proposition}
The linear recurrence constant is of the Tribonacci word $\infw{t}$ is $2\psi^2 + \psi + 1$, where $\psi\approx 1.839$ is the dominant root of the polynomial $X^3-X^2-X-1$.
\end{proposition}
\begin{proof}
As $i\rightarrow \infty$, it is well-known that we have $T_i\sim  c_T\psi^{i+2}$ with $c_T=\frac{1}{-\psi^2+4\psi-1}$, and similarly one can check that $U_i\sim c_U \psi^{i+6}$ with $c_U=\frac{1}{4\psi^2+6\psi+4}$. Using the previous lemma and \eqref{eq:linrec L+1}, the linear recurrence constant is given by
\[
\lim_{i \rightarrow +\infty} \frac{T_{i+4}}{U_i + 1}= \frac{c_U}{c_T}=2\psi^2 + \psi + 1 \approx 9.605,
\]
which is enough.
\end{proof}

\section{Details on \texorpdfstring{\cref{exa:nonrecurrent_FB}}{} } \label{appendix:example}

Let us recall that the useful definitions from~\cref{exa:nonrecurrent_FB}. 
We let $\tau \colon \{0,1\}^* \to \{0,1\}^* \colon 0 \mapsto 01, 1\mapsto 0$ and $\sigma \colon \{0,1,*\}^* \to \{0,1,*\}^*$ be defined by $\sigma(a) = \tau(a)$ for $a \in \{0,1\}$ and $\sigma(*) = * 11$.
Let $\infw{x}= \sigma^\omega(*) = *11000101010010010010100\cdots$ and consider $\infw{x'}=11000\cdots$, the first shift of $\infw{x}$.
We prove several properties about these words.

\begin{lemma}
The words $\infw{x}$ and $\infw{x'}$ are both letter-balanced.
\end{lemma}
\begin{proof}
Let us prove that the shift $\infw{x'}$ of $\infw{x}$ is letter-balanced, which clearly implies that $\infw{x}$ is letter-balanced too.
It is clear that the word $\infw{x'}$ admits letter frequencies, and these frequencies are respectively $\mu_0=1/\varphi$ and $\mu_1=1/\varphi^2$, where $\varphi=(1+\sqrt{5})/2$ is the golden ratio (note that these are the same frequencies as in the Fibonacci word).
Then, let us only prove that the discrepancy of $0$ in any factor is finite, since the proof regarding the discrepancy of $1$ is similar. Let $w$ be a factor of $\infw{x'}$, and let us prove that $||w|_0-\mu_0|w||$ is bounded. Write $w=p_{w}\tau^{\ell_1}(1)\tau^{\ell_2}(1)\cdots \tau^{\ell_k}(1)s_w$, for some integer sequence $(\ell_i)_{1\le i\le k}$ and $p_w,s_w$ factors of the Fibonacci word. By the classical Perron-Frobenius theory, it is well-known that we have $|\tau^n(1)|_0-\mu_0|\tau^n(1)|=O(\bar{\varphi}^n)$, with $\bar{\varphi}=(1-\sqrt{5})/2$. 



Since the Fibonacci word is letter-balanced, we have $|p_w|_0-\mu_0|p_w|=O(1)$ and $|s_w|_0-\mu_0|s_w|=O(1)$. Therefore, we have \begin{align*}
    ||w|_0-\mu_0|w||&=\Big||p_w|_0+\sum_{1\leq i \leq k}|\tau^{\ell_i}(1)|_0+|s_w|_0-\mu_0\Big(|p_w|+\sum_{1\leq i \leq k}|\tau^{\ell_i}(1)|+|s_w|\Big)\Big| \\
    &=\Big|\sum_{1\leq i \leq k}O(\bar{\varphi}^{\ell_i})+O(1)\Big|,
\end{align*} and the sum $\sum_{1\leq i \leq k}O(\bar{\varphi}^{\ell_i})$ converges as $\ell_k\rightarrow +\infty$ since $|\bar{\varphi}|<1$.  
\end{proof}

Like the Tribonacci case in the previous appendix, the \emph{Zeckendorf numeration system} is a positional numeration system for integers based on the sequence of \emph{Fibonacci numbers}.
In this numeration system, a sequence is said to be \emph{Fibonacci-automatic} if it can be generated by a DFAO where the input is read in the Zeckendorf numeration system. 


\begin{lemma}
The word $\infw{x}$ is Fibonacci-automatic and is generated by the DFAO of~\cref{fig:DFAO for twisted Fibonacci} in the Zeckendorf numeration system. 
\end{lemma}

\begin{figure}[ht]
\includegraphics[width=\textwidth]{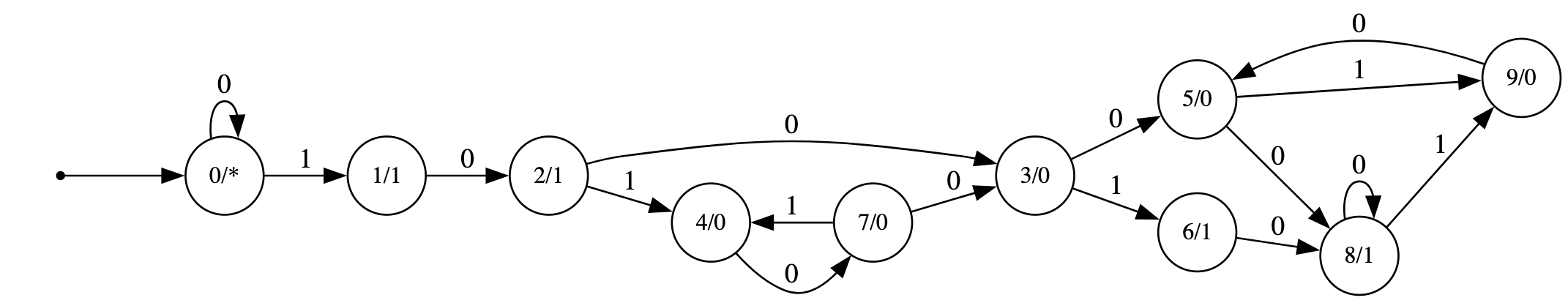}
\caption{A DFAO generating $\infw{x}$ in the Zeckendorf numeration system.}
\label{fig:DFAO for twisted Fibonacci}
\end{figure}

\begin{proof}
 It is enough to prove that the automaton in~\cref{fig:DFAO for twisted Fibonacci} generates $\infw{x}$.
Let $\cA = \{*,0,1\}$ be the alphabet of the substitution $\sigma$ and let $\cC = \{0,1,2,3,4,5,6,7,8,9\}$ be the set of states of the automaton in~\cref{fig:DFAO for twisted Fibonacci}.
Define the substitutions $\psi \colon \cC^* \to \cC^*$ and $\pi \colon \cC^* \to \cA^*$ by the automaton, i.e.,
\[
\psi \colon 
0 \mapsto 01,\, 
1 \mapsto 2, \,
2 \mapsto 34, \,
3 \mapsto 56, \,
4 \mapsto 7, \,
5 \mapsto 89, \,
6 \mapsto 8, \,
7 \mapsto 34, \,
8 \mapsto 89, \,
9 \mapsto 5,
\]
and
\[
\pi \colon 
0 \mapsto *, \,
1 \mapsto 1, \,
2 \mapsto 1, \,
3 \mapsto 0, \,
4 \mapsto 0, \,
5 \mapsto 0, \,
6 \mapsto 1, \,
7 \mapsto 0,\,
8 \mapsto 1,\,
9 \mapsto 0.
\]
We claim that the infinite word generated by the automaton is $\infw{x}$, i.e., $\pi(\psi^\omega(0)) = \infw{x} = \sigma^\omega(*)$.
Define the substitutions $\phi,\delta \colon \cC^* \to \cA^*$ by  $\phi = \pi \circ \psi$ and
\[
\delta \colon
0 \mapsto \varepsilon, \,
1 \mapsto 0, \,
2 \mapsto \varepsilon, \,
3 \mapsto 1, \,
4 \mapsto 0, \,
5 \mapsto 1, \,
6 \mapsto 1, \,
7 \mapsto \varepsilon, \,
8 \mapsto 1, \,
9 \mapsto 1.
\]
Observe that
\[ 
\pi(\psi^2(a))\, \delta(b) = 
\delta(a) \, \sigma(\phi(a))
\]
for all $ab \in \cL(\psi^\omega(0)) \cap \cC^2 = \{01,12,
23,34,45,48,56,58,67,78,83,89,95,98\}$. 
Therefore, by an inductive argument, for all $\ell \ge 1$
\begin{align}
\label{diagrama_chueco_automata}
\pi(\psi^2(a_0\cdots a_{\ell-1}))\, \delta(a_\ell) =
\delta(a_0) \, \sigma(\phi(a_0\cdots a_{\ell-1}))
\end{align}
for all $a_0\cdots a_\ell \in \cL(\psi^\omega(0)) \cap \cC^{\ell+1}$.


Let $k \ge 1$ be an integer, set $a_0\cdots a_{\ell-1} = \psi^k(0)$ and let $a_\ell \in \cC $ be such that $\psi^k(0)a_\ell \in \cL(\psi^\omega(0)) \cap \cC^{\ell+1}$. Since the word $\psi^k(0)$ begins with $0$ and $\delta(0)=\varepsilon$, Equality~\eqref{diagrama_chueco_automata} implies that
$\pi(\psi^{k+2}(0)) \delta(a_\ell)=\sigma(\phi(\psi^k(0)))$, so $\pi(\psi^{k+2}(0))$ is a prefix of $\sigma(\phi(\psi^k(0)))$.
Taking the limit as $k \to +\infty$ and setting $\infw{y} := \psi^\omega(0)$, we obtain $\pi(\infw{y}) = \sigma(\phi(\infw{y}))$.
Since $\phi = \pi \circ \psi$ and $\psi(\infw{y}) = \infw{y}$, it follows that
$\infw{z} \coloneqq \pi(\infw{y})$ satisfies $\infw{z} = \sigma(\infw{z})$.
Thus, $\infw{z} = \sigma^m(\infw{z})$ for all $m \in \NN$.
Note that $\infw{z}$ starts with $*$, since $\infw{y}$ starts with $0$ and $\pi(0) = *$.
Therefore, by taking the limit $m \to +\infty$, we conclude that
$\infw{z} = \sigma^\omega(*) = \infw{x}$, as desired.
\end{proof}

\begin{remark}
Notice that it is not trivial at all that $\infw{x}$ is Fibonacci-automatic. Indeed, the polynomial of the substitution $\sigma$ is $(X-1)(X^2-X-1)$, which does not guarantee it to be Fibonacci-automatic in general.      
\end{remark}

We use the free software \texttt{Walnut} to prove the next statement.

\begin{proposition}
The word $\infw{x}$ is $5$-power-free and no shift of $\infw{x}$ is recurrent.  
\end{proposition}
\begin{proof}
To prove the first part of the statement, the following two commands test whether the word $\infw{x}$ contains
$4$-or $5$-powers, respectively:
\begin{verbatim}
eval pow4 "?msd_fib Ei,n n>0 & At (t>=i & t<i+3*n)=> X[t]=X[t+n]":
eval pow5 "?msd_fib Ei,n n>0 & At (t>=i & t<i+4*n)=> X[t]=X[t+n]":
\end{verbatim}
The first returns \texttt{True} and the second returns \texttt{False}, showing that
$\infw{x}$ contains $4$-powers but is $5$-power-free.

In \texttt{Walnut}, we define equality of factors as follows: the predicate \texttt{eqfac}$(i,j,n)$ returns \texttt{True} if and only if $\infw{x}[i..i+n] = \infw{x}[j..j+n]$, which is defined with the following command:
\begin{verbatim}
def xeqfac "?msd_fib Au,v (u>=i & u<i+n & u+j=v+i) => X[u]=X[v]":
\end{verbatim}
This produces an automaton with $199$ states. To prove the second part of the statement, we use the following command:
\begin{verbatim}
def norec "?msd_fib An Ei,m i>n & m>n & Aj (j!=i)=> ~$xeqfac(i,j,m)":
\end{verbatim}
This returns \texttt{True}, showing that no shift of $\infw{x}$ is recurrent.    
\end{proof}

\end{document}